\documentclass[10pt]{amsart}
\usepackage{amssymb}

\usepackage{graphicx}
\usepackage{xcolor}

\newcommand{\res}{\upharpoonright}
\newcommand{\A}{{\mathbb A}}

\theoremstyle{plain}
\newtheorem{theorem}{Theorem}[section]

\newtheorem{lemma}[theorem]{Lemma}
\newtheorem{proposition}[theorem]{Proposition}

\newtheorem*{definition}{Definition}

\theoremstyle{remark}

\numberwithin{equation}{section}

\newcommand{\dbullet}{\,\raisebox{-1.4pt}{\( \bullet \)}\,}

\begin{document}
\title{Dual Ramsey theorem for trees}

\author{S{\l}awomir Solecki}

\thanks{Research supported by NSF grants DMS-1266189 and DMS-1800680.} 

\address{Department of Mathematics\\
Cornell University\\
Malott Hall\\
Ithaca, NY 14853}

\email{ssolecki@cornell.edu}



\begin{abstract}
The classical Ramsey theorem was generalized in two major ways: to the dual Ramsey theorem, by Graham and Rothschild, and to Ramsey theorems for trees,
initially by Deuber and Leeb. Bringing these two lines of thought together, we prove the dual Ramsey theorem for trees. Galois connections between partial orders
are used in formulating this theorem, while the abstract approach to Ramsey theory, we developed earlier, is used in its proof.
\end{abstract}

\maketitle

\section{Introduction}

A rich theory of Ramsey results has been developed since the publication of Ramsey's original paper. (For
an introduction to the subject see \cite{Nes}.)
The discovery in \cite{KPT} of close connections between Ramsey Theory and Topological Dynamics gave rise to substantial new advances in the theory in the last decade.
(The reader may consult \cite{NVT} for a survey.) The present paper was motivated in equal measure by these recent developments and by the internal logic of Ramsey Theory
as it relates to the idea of duality. (For a different aspect of duality in Ramsey Theory, see \cite{So0}.)

The Dual Ramsey Theorem was proved by Graham and Rothschild in \cite{GrRo71}. It
was then realized that the dual version was, in fact, a strengthening of Ramsey's original result.
Another independent line of generalizations of Ramsey's theorem was initiated
by Deuber \cite{Deu} and Leeb, see \cite{GrRo75}. These authors generalized Ramsey's theorem from linear orders to trees. Further Ramsey theorems for trees
were found in \cite{Fou}, \cite{Jas}, \cite{Mil} (see also \cite{Soc}), and \cite{So2}. (Paper \cite{So2} provides a uniform treatment of these results.)

The aim of the present paper is to bring together these two lines of development by proving the Dual Ramsey Theorem for Trees as announced in \cite{So3}. This theorem is
a common strengthening of two classical results---Leeb's Ramsey theorem for trees and Graham and Rothschild's Dual Ramsey Theorem.
It should be noted that the first one of these theorems is formulated in terms of copies of trees, the second one in terms of partitions of finite initial segments
of natural numbers. So the first challenge is to find objects that generalize both: copies of trees and partitions. To this end, the two classical Ramsey
theorems are restated in terms of functions. Their common generalization is then formulated using functions that turn out to come from appropriately modified
Galois connections in the sense of Ore \cite{Ore}, \cite{G-S}. (The association of duality in Ramsey theory with Galois connections
is new and may be worth further investigation.) This generalization, which is the main theorem of the paper,
is then proved with the use of our abstract approach to Ramsey theory from \cite{Sol}.

Aside from the theoretical considerations, the motivation for our main result comes, in a vague sense,
from the recent results in \cite{BaKw} and \cite[Section 3]{Mo}.

In Section~\ref{S:tree}, we give all the required definitions,
the statement of our main result, Theorem~\ref{T:mn}, and its context. We also prove there that the main theorem
strengthens the two classical Ramsey results mentioned above.
In Section~\ref{S:aR}, we outline the fragment of the abstract
Ramsey theory developed in \cite{Sol} that is needed for our proof and we state the appropriate versions of
the Hales--Jewett theorem that will be used. In Section~\ref{S:proof}, we give a proof of the main result; its principal technical 
argument is contained in Section~\ref{Su:lpP}.

\section{The theorem and its context}\label{S:tree}

We start this section with collecting the basic notions concerning trees.
Then we state our main definition of rigid surjections between trees and formulate the main result---the Ramsey theorem for
rigid surjections, which we call the Dual Ramsey Theorem for Trees. We follow it with a restatement of two classical Ramsey theorems---Leeb's
Ramsey theorem for trees and Graham and Rothschild's Dual Ramsey Theorem. We show that rigid surjections between trees are objects that are more
general than the objects in these two classical Ramsey statements, and we give an argument that the Dual Ramsey Theorem for Trees is their common
generalization. We finish this section with explaining how rigid surjections fit in the larger framework of
Galois connections.

\subsection{Ordered trees}\label{Su:deot}

By a {\em tree} $T$
we understand a finite, partial ordered set with a smallest element, called {\em root}, and such that
the set of predecessors of each element is linearly ordered. So in this paper, {\em all trees are non-empty and finite}.
By convention, we regard every node of a tree as one of its
own predecessors and as one of its own successors. We denote the tree order on $T$ by
\[
\sqsubseteq_T.
\]

Each tree $T$ carries a binary function $\wedge_T$ that assigns to each
$v,w\in T$ the largest with respect to $\sqsubseteq_T$ element $v\wedge_Tw$ of $T$ that is a predecessor of
both $v$ and $w$.

For a tree $T$ and $v\in T$, let ${\rm im}_T(v)$ be the set of all {\em immediate successors} of $v$, and
we do not regard $v$ as one of them.
(We will occasionally suppress the subscripts from various pieces of notation introduced above if we
deem them clear from the context.)
A tree $T$ is called {\em ordered} if for each $v\in T$ there is a fixed linear order of ${\rm im}(v)$. Such an assignment allows us to define
the lexicographic linear order
\[
\leq_T
\]
on all the nodes of $T$ by stipulating that $v\leq_T w$ if $v$ is a predecessor of $w$ and, in case $v$ is not a predecessor
of $w$ and $w$ is not a predecessor of $v$, that $v\leq_Tw$
if the predecessor of $v$ in ${\rm im}(v\wedge w)$ is less than or equal to the predecessor of $w$ in
${\rm im}(v\wedge w)$ in the given order on ${\rm im}(v\wedge w)$.

\subsection{The notion of rigid surjection}

The following definition is essentially due to Deuber \cite{Deu}. Let $S$ and $T$ be ordered trees. A function $e\colon S\to T$ is called a {\em morphism} if
\begin{enumerate}
\item[(i)] for $v,w\in S$,
\[
e(v\wedge_Sw) = e(v)\wedge_T e(w);
\]

\item[(ii)] $e$ is monotone between $\leq_S$ and $\leq_T$, that is, for $v,w\in S$,
\[
v\leq_S w\Longrightarrow e(v)\leq_T e(w);
\]

\item[(iii)] $e$ maps the root of $S$ to the root of $T$.
\end{enumerate}

An {\em embedding} is an injective morphism.

Here is the definition of functions for which our main theorem will be proved. As explained in Section~\ref{Su:Gal}, it comes from
the notion of Galois connection.

\begin{definition}\label{D:rs} Let $S$, $T$ be ordered trees. A function $f\colon T\to S$ is called a {\em rigid surjection} provided
there exists a morphism $e\colon S\to T$ such that
\begin{equation}\label{E:der}
f\circ e = {\rm id}_S\;\hbox{ and }\; e\circ f \sqsubseteq_T {\rm id}_T.
\end{equation}
\end{definition}
The last condition in the definition means that $e(f(w))\sqsubseteq_T w$ for each $w\in T$. Note that $f$ need not be a morphism.
It is clear from the definition that $f$ is surjective and $e$ injective, so $e$ is an embedding.

We note that in the above situation $f$ determines $e$, that is, if $f\colon T\to S$ and
$e_1, e_2$ are morphisms from $S$ to $T$ such that \eqref{E:der} holds for each of them, then $e_1=e_2$.
(This means that $e$ can be defined from $f$; indeed, if
$f \colon T \to S$ is a rigid surjection, then $e \colon S \to T$ is given by $e(v) = \bigwedge_T f^{-1}(v)$.)
We call this unique $e$ the {\em injection of }$f$.

We register the following easy to prove lemma.

\begin{lemma}\label{L:coin}
Let $f\colon T\to S$ and $g\colon V\to T$ are rigid surjections, then so is $f\circ g$. In fact, if $d$ and $e$ are the injections of $f$ and $g$,
respectively, then $e\circ d$ is the injection of $f\circ g$.
\end{lemma}

We also have the following lemma.

\begin{lemma}\label{L:prexmb}
Let $S$ and $T$ be ordered trees. Let $e\colon S\to T$ be an embedding. There exits a rigid surjection $f\colon T\to S$
such that $e$ is the injection of $f$.
\end{lemma}

\begin{proof} For $w\in T$, define $f(w)$ to be the $\sqsubseteq_S$-largest $v\in S$ such that $e(v)\sqsubseteq_T w$. We leave
checking that this $f$ works to the reader.
\end{proof}

Observe that, in general, there are many rigid surjections with the same injection.

\subsection{The main theorem}

By a {\em $b$-coloring}, for a natural number $b>0$, we understand a coloring with $b$ colors. The following result is the main theorem of the paper.

\begin{theorem}\label{T:mn}
Let $b$ be a positive integer. Let $S,T$ be ordered trees. There exists an ordered tree $U$ such that for each $b$-coloring
of all rigid surjections from $U$ to $S$ there is a rigid surjection $g_0\colon U\to T$ such that
\[
\{ f\circ g_0\mid f\colon T\to S \hbox{ a rigid surjection}\}
\]
is monochromatic.
\end{theorem}

\subsection{Ramsey theorem for trees and Dual Ramsey Theorem as consequences of Theorem~\ref{T:mn}}

An image of a tree $S$ under
an embedding from $S$ to $T$ is called a {\em copy of $S$ in $T$}. The following theorem is due to Leeb, see \cite{GrRo75}. (Sometimes, the theorem below is formulated with a weaker definition of embedding, in which
condition (iii) is omitted. The two versions are easily derivable from each other.)

\smallskip

\noindent {\em Given a positive integer $b$ and ordered trees $S$ and $T$, there is an ordered tree $U$ such that for each $b$-coloring of all copies of $S$ in $U$
there is a copy $T'$ of $T$ in $U$ such that all copies of $S$ in $T'$ get the same color.}

\smallskip

We chose to formulate this theorem directly in terms of embeddings.

\begin{theorem}[Leeb]\label{T:lee} Let $b$ be a positive integer. Let $S$ and $T$ be ordered trees. There exists an
ordered tree $U$ such that for each $b$-coloring of all embeddings from $S$ to $U$, there exists an embedding
$e_0\colon T\to U$ such that
\[
\{ e_0\circ d\mid d\colon S\to T\hbox{ an embedding}\}
\]
is monochromatic.
\end{theorem}

To derive the above theorem from Theorem~\ref{T:mn}, given $S$ and $T$ and the number of colors, let $U$ be the ordered tree from Theorem~\ref{T:mn}.
This $U$ works also for Theorem~\ref{T:lee}. Indeed, given a coloring of all embeddings from $S$ to $T$, we assign a rigid surjection from $T$ to $S$ the color of its
injection. Theorem~\ref{T:mn} produces a rigid surjection $g_0\colon U\to S$. Let $e_0$ be the injection of $g_0$. It is easy to check, using Lemma~\ref{L:prexmb}, that the conclusion of Theorem~\ref{T:lee} holds for it.

For a natural number $n$, let $[n]$ stand for $\{ 1, \dots, n\}$.
The following is the dual Ramsey theorem of Graham and Rothschild \cite{GrRo71}.

\smallskip

\noindent {\em Given a positive integer $b$ and positive integers $k,\, l$ there exists a positive integer $m$ such that for each $b$-coloring
of all $k$ element partitions of $[m]$ there exists an $l$ element partitions $Q$ of $[m]$ such that all $k$ element partitions of $[m]$ that are
coarser than $Q$ have the same color.}

\smallskip

It was noticed already by Pr{\"o}mel and Voigt \cite{PrVo} that a restatement of the dual Ramsey theorem in terms of functions was possible.
They called s function $f\colon [n]\to [m]$ a rigid surjection if $f$ is surjective and, for each $y\in [n]$,
\[
f(y) \leq 1+ \max_{x<y} f(x)
\]
with the convention that $\max$ over the empty set is $0$.
Note that sets of the form $[n]$ for $n\in {\mathbb N}$ with their natural inequality relation and the unique
ordering of the immediate successors of each vertex are ordered trees. In fact, the tree relation and $\sqsubseteq_{[n]}$ and the
linear order relation $\leq_{[n]}$ are equal to each other. By treating $[m]$ and $[n]$ as ordered trees
$f\colon [n]\to [m]$ is a rigid surjection according to the above definition
precisely when it is a rigid surjection according to our definition of rigid surjection between trees. Indeed,
$f\colon [n]\to [m]$ that is a rigid surjection according to the above definition, the function $e\colon [m]\to [n]$ given
by $e(x) = \min f^{-1}(x)$ witnesses that $f$ is a rigid surjection according to our definition.

\begin{theorem}[Graham--Rothschild]\label{T:grro}
Let $b$ be a positive integer. Given $k$ and $l$, there exists $m$ such that for each $b$-coloring of all rigid surjections
from $[m]$ to $[k]$ there is a rigid surjection $g_0\colon [m]\to [l]$ such that
\[
\{ f\circ g_0\mid f\colon [l]\to [k]\hbox{ a rigid surjection}\}
\]
is monochromatic.
\end{theorem}

To see how Theorem~\ref{T:grro} follows from Theorem~\ref{T:mn}, apply Theorem~\ref{T:mn} to the ordered trees $S= [k]$ and $T=[l]$ obtaining
an ordered tree $U$. Then $U$ with its linear ordering $\leq_U$ is isomorphic as a linear order to some $[m]$. For this $m$ the conclusion of Theorem~\ref{T:grro}
holds. This is immediate once we observe that a rigid surjection from $U$ to $[l]$ is also a rigid surjection from the linear order $(U, \leq_U)$, that is from $[m]$, to $[l]$.

\subsection{The context for rigid surjections---Galois connections}\label{Su:Gal}

Let $(S, \sqsubseteq_S)$ and $(T, \sqsubseteq_T)$ be two partial orders, not necessarily trees, for now.
A pair $(f,e)$ is called a {\em Galois connection} if $f\colon T\to S$, $e\colon S\to T$, and both
\begin{equation}\label{E:gin}
f\circ e \sqsubseteq_S {\rm id}_S \;\hbox{ and }\; e\circ f \sqsubseteq_T {\rm id}_T
\end{equation}
Galois connections in their abstract form were first defined by Ore in \cite{Ore}, and we essentially followed the original definition.
(Usually both $e$ and $f$ are assumed to be monotone, but we will need the broader notion here.)
For a comprehensive treatment see \cite{G-S}. As already noticed by Ore, of particular
importance are Galois connections for which equality holds in one of the inequalities in \eqref{E:gin}; such Galois connections are called perfect in \cite{Ore}.
So we are interested in Galois connections fulfilling
\begin{equation}\label{E:gin2}
f\circ e = {\rm id}_S\;\hbox{ and }\; e\circ f \sqsubseteq_T {\rm id}_T.
\end{equation}
Galois connections with \eqref{E:gin2} are often called embedding--projection pairs. They
are important in denotational semantics of programming languages, see for example \cite{DrGo}, and are relevant in some topological considerations, see
for example \cite{Ku}.

Now we consider \eqref{E:gin2} and assume that $S$ and $T$ are ordered trees.

Assuming that $f$ is a morphism puts restrictions on $e$; it is easy to see that it implies that $e$ is a morphism as well. 
Moreover, $f$ determines $e$ and $e$ determines $f$. So formulating the Ramsey statement for this kind of functions, we get Leeb's Ramsey result; 
if stated for $e$, it takes the form of Theorem~\ref{T:lee},
if stated for $f$, it takes the equivalent surjective form.

On the other hand, $e$ being a morphism does not put severe restrictions on $f$, in particular, it does not imply that $f$ is a morphism. 
In this case, $f$ is what we called a rigid surjection. The Ramsey theorem stated for such functions $f$ is our main result.

\section{The tools: abstract Ramsey theory and pigeonhole lemmas}\label{S:aR}

Theorem~\ref{T:mn} will be proved using the abstract approach to Ramsey theory developed in \cite{Sol}. In Sections~\ref{Su:ncs} and \ref{Su:bat},
we present a fragment of this approach that is sufficient for our goals here. The abstract Ramsey theorem is stated as Theorem~\ref{T:ram}. The main
difficulty in applying this theorem in concrete situations is deducing the abstract pigeonhole condition (LP).
To achieve this in our situation in later sections, we will need certain
known Hales--Jewett--type results, which we collect in Section~\ref{Su:aux}.

\subsection{Normed composition spaces}\label{Su:ncs}

Let $\A$ be a set. Assume we are given a {\em partial} function from $\A\times \A$ to $\A$:
\[
(a, b)\to a\cdot b,
\]
which is associative, that is, for $a,b,c\in \A$ if $a\cdot (b\cdot c)$ and
$(a\cdot b)\cdot c$ are both defined, then
\begin{equation}\label{E:acct}
a\cdot (b\cdot c) = (a\cdot b)\cdot c.
\end{equation}
We assume we also have a function $\partial\colon \A\to \A$ and a function $|\cdot|\colon \A\to L$, where $L$ is equipped with a partial order $\leq$.

A structure as above is called a {\em normed composition space} if the following conditions hold for $a, b, c\in \A$:
\begin{enumerate}
\item[(i)] if $a\cdot b$ and $a\cdot\partial b$ are defined, then
\begin{equation}\notag
\partial (a\cdot b) = a\cdot \partial b\, ;
\end{equation}

\item[(ii)] $|\partial a|\leq |a|$;

\item[(iii)] if $|b|\leq |c|$ and $a\cdot c$ is defined, then $a\cdot b$ is defined and $|a\cdot b|\leq |a\cdot c|$.
\end{enumerate}

The operation $\cdot$ is called a {\em multiplication}. We call $\partial$ a {\em truncation} and $|\cdot|$ a {\em  norm}.

Given $a,b\in \A$, we say that $b$ {\em extends} $a$ if for each $x\in \A$ with
$a\cdot x$ defined, we have that $b\cdot x$ is defined and that it is equal to $a\cdot x$.

For $t\in {\mathbb N}$, we write $\partial^t$ for the $t$-th iteration of $\partial$. For a subset $P$ of $\A$, we write $\partial P = \{ \partial a\mid a\in P\}$.

\subsection{Ramsey domains}\label{Su:rd}

Let ${\mathcal F}$ and ${\mathcal P}$ be families of non-empty subsets of $\A$.
Assume we have a partial function $\bullet$ from ${\mathcal F}\times {\mathcal F}$ to $\mathcal F$
with the property that if $G\bullet F$ is defined, then it is given point-wise, that is,
$f\cdot g$ is defined for all $f\in F$ and $g\in G$, and
\[
F\bullet G = \{ f\cdot g\colon f\in F,\, g\in G\}.
\]
Assume we also have a partial function from ${\mathcal F}\times {\mathcal P}$ to
${\mathcal P}$, $(F, P)\to F\dbullet P$, such that if $F\dbullet P$ is defined, then $f\cdot x$ is defined for all $f\in F$
and $x\in P$ and
\[
F\dbullet P = \{ f\cdot x\colon f\in F,\, x\in P\}.
\]

The structure $({\mathcal F}, {\mathcal P}, \dbullet, \bullet)$ as above is called a {\em Ramsey domain} over the normed composition space
$(\A, \cdot, \partial, |\cdot |)$ if sets in $\mathcal P$ are finite and the following conditions hold:
\begin{enumerate}
\item[{\bf (A)}] if $F,G\in {\mathcal F}$, $P\in {\mathcal P}$, and
$F\dbullet (G\dbullet P)$ is defined, then so is $(F\bullet
G)\dbullet P$;

\item[{\bf (B)}] if $P\in {\mathcal P}$, then $\partial P\in {\mathcal
P}$;

\item[{\bf (C)}] if $F\in {\mathcal F}$, $P\in {\mathcal P}$, and
$F\dbullet\partial P$ is defined, then there is $G\in {\mathcal
F}$ such that $G\dbullet P$ is defined and for each $f\in F$
there is $g\in G$ extending $f$.
\end{enumerate}

A Ramsey domain as above is called {\em vanishing} if for each $P\in {\mathcal P}$ there is $t\in {\mathbb N}$ such that
$\partial^t P$ has only one element. It is called {\em linear} if $\{ |x|\colon x\in P\}$ is a linear subset of $L$ for each $P\in {\mathcal P}$.

\subsection{Abstract Ramsey theorem}\label{Su:bat}

The following condition is our Ramsey statement:
\begin{enumerate}
\item[{\bf (R)}] given a natural number $b>0$, for each $P\in {\mathcal P}$, there is an
$F\in {\mathcal F}$ such that $F\dbullet P$ is defined, and for
every $b$-coloring of $F\dbullet P$ there is an $f\in F$ such that
$f\cdot P$ is monochromatic.
\end{enumerate}

For $P\subseteq \A$ and $y\in \A$, put
\begin{equation}\notag
P^{y} = \{ x\in P \mid \partial x = y\}.
\end{equation}
For $F\subseteq \A$ and $a\in \A$, let
\begin{equation}\notag
F_a = \{ f\in F\mid f\hbox{ extends }a\}.
\end{equation}

The following criterion is our pigeonhole principle:
\begin{enumerate}
\item[{\bf (LP)}] given a natural number $b>0$, for all $P\in {\mathcal P}$ and $y\in
\partial P$, there are $F\in {\mathcal F}$ and $a\in \A$ such
that $F\dbullet P$ is defined, $a\cdot y$ is defined, and for
every $b$-coloring of $F_a\cdot P^{y}$ there is an $f\in F_a$ such that
$f\cdot P^{y}$ is monochromatic.
\end{enumerate}

The theorem below is the main abstract Ramsey theorem stating that,
under appropriate conditions, the pigeonhole principle implies the Ramsey statement. It is proved in \cite[Theorem~5.3]{Sol}.

\begin{theorem}\label{T:ram}
Let $({\mathcal F}, {\mathcal P},\bullet, \dbullet)$ be a vanishing linear Ramsey domain over a normed composition space.
Then (LP) implies (R).
\end{theorem}

\subsection{Concrete pigeonhole lemmas}\label{Su:aux}

We formulate here lemmas that will be used to prove condition (LP) for the concrete Ramsey domain defined later. They are restatements of known results.

The first lemma is a version of the Hales--Jewett theorem in disguise; for a proof apply statement (HL2) with $t=1$
from \cite[Appendix 2]{So2} and the standard pigeonhole principle. (The lemma follows also from Leeb's theorem stated as Theorem~\ref{T:lee} in the present paper, 
but it is simpler than this theorem, 
as it is a rephrasing of the Hales--Jewett theorem stated as in \cite[Appendix 2]{So2}.) 

\begin{lemma}\label{L:prfrco}
Let $b>0$. Let $S$ be an ordered tree and let $v_0$ be its root. There exists an ordered tree $S'$ such that for each $b$-coloring of vertices of $S'$
there is an embedding $i\colon S\to S'$ such that all elements of $i(S\setminus \{ v_0\})$ have the same color.
\end{lemma}

For linear orders $A$ and $L$, let
\[
A\oplus L
\]
be the linear order obtained by putting the linear order of $L$ on top of the linear order of $A$.
We consider $A$ and $L$ to be included in $A\oplus L$. Let
\[
A\oplus 1
\]
stand for $A\oplus L$, where $L$ is the linear order consisting of one element.

Fix linear orders $A$, $L$, and $I$. We consider $L\times I$ as linearly ordered by the lexicographic order.
For a function
\[
p\colon A\oplus (L\times I) \to A\oplus L
\]
we will be interested in the following property
\begin{equation}\label{E:spfu}
p\res A = {\rm id}_A\;\hbox{ and }\;
\forall x\in L\; x\in p[\{ x\}\times I] \subseteq A\cup \{ x\}.
\end{equation}
Each such $p$ is a rigid surjection. To clarify condition \eqref{E:spfu}, note that
for each $x\in L$, the set $\{ x\}\times I$ is an interval
in the linear order $L\times I$. The second part of condition \eqref{E:spfu} says that on that interval
the only values possibly attained by $p$ are $x$ and
points in $A$, and $x$ is actually attained.

For an element $x$ of a linear order, let
\[
x-
\]
stand for the immediate predecessor of $x$, if there is one. For a linear order $L$ and $x\in L$,
let
\begin{equation}\label{E:cutt}
L^x
\end{equation}
stand for the linear order on $L$ restricted to the set $\{ y\in L\mid y\leq_L x\}$.

We use the above notions to isolate, in Lemma~\ref{L:prHJ} below, the version of the Hales--Jewett theorem we will need. It
is a version of the left-variable word Hales--Jewett theorem. This particular statement is essentially proved in
\cite[Section 8.1]{Sol}. We will explain it precisely in the proof below.

\begin{lemma}\label{L:prHJ}
Let $b>0$. Let two linear orders $A$ and $L$ be given with $A$ non-empty. There is a linear order $I$ such that for each $b$-coloring
of all functions from $(A\oplus (L\times I))^{y-}$ to $A$, that are identity on $A$ and where we allow $y$ to vary over $L\times I$, there is
\[
p\colon A\oplus  (L\times I) \to A\oplus  L
\]
with property \eqref{E:spfu} and such that the color of
\[
r\circ (p\res \{ z\in A \oplus (L\times I) \colon z<_{A \oplus (L\times I)}  \min p^{-1}(x)\}),
\]
where $r\colon (A\oplus L)^{x-} \to A$ and $r\res A = {\rm id}_A$,
depends only on $x\in L$.
\end{lemma}

\begin{proof} Given linear orders $I_y$, for $y\in L$, where $L$ is a linear order, let $\bigoplus_{y\in L} I_y$ be the linear order
on the disjoint union $\bigcup_{y\in L} I_y$ that on each set $I_y$ coincides with the order with which this set is equipped and makes all
elements of $I_y$ smaller that all elements of $I_{y'}$ if $y<_Ly'$.

An inspection of the proof of the Hales--Jewett theorem in \cite[Section~8.1, Lemma~8.1]{Sol} reveals that the following
statement is proved there.

\noindent For $b>0$ and two linear orders $A$ and $L$ with $A$ non-empty,
there exist linear orders $I_y$, for $y\in L$, such that for each
$b$-coloring of all functions from $(A \oplus \bigoplus_{y\in L} I_y)^{x-}$ to $A$ that are identity on $A$,
with $x\in \bigoplus_{y\in L} I_y$,
there is a function $p\colon A \oplus \bigoplus_{y\in L} I_y \to A\oplus L$ such that $y\in p(I_y) \subseteq A\cup \{ y\}$ and, for each
$r\colon (A\oplus L)^{x-}\to A$, with $x\in L$ and with $r\res A = {\rm id}_A$, the color of
\[
r\circ (p\res \{ z\in A \oplus \bigoplus_{y\in L} I_y \colon z<_{(A \oplus \bigoplus_{y\in L} I_y)} \min p^{-1}(x)\})
\]
depends only on $x$.

It is clear that we can make $I_y= I$, for some linear order $I$ and for
all $y\in L$, by enlarging each of them to the size of the largest linear order among the $I_y$-s. So we have $\bigoplus_{y\in L} I_y  = L\times I$, as needed
in the conclusion of the lemma.
\end{proof}

\section{The proof of Theorem~\ref{T:mn}}\label{S:proof}

In this section, first, we apply the abstract approach as outlined in Section~\ref{S:aR} to prove Proposition~\ref{P:almo}, which is a version of
Theorem~\ref{T:mn} for a certain subclass of rigid surjections and which may be of some independent interest. Then we
deduce full Theorem~\ref{T:mn} from this particular case. One of the technically important points in applying the abstract approach is finding truncation operations.
We find two truncations, one in Section~\ref{Su:ras}, the other one in Section~\ref{Su:lot}. The first one will be used to prove Proposition~\ref{P:almo},
the second one to carry over the result to arbitrary rigid surjections in Theorem~\ref{T:mn}.

In Section~\ref{Su:ras}, we introduce the particular type of rigid surjections, we call sealed, and we state, as Proposition~\ref{P:almo}, a result analogous
to Theorem~\ref{T:mn} for such rigid surjections. In Sections~\ref{Su:raP} and \ref{Su:lpP}, we prove Proposition~\ref{P:almo}. Then in Section~\ref{Su:setsu} we
derive Theorem~\ref{T:mn} from Proposition~\ref{P:almo}.

\subsection{A Ramsey result for sealed rigid surjections}\label{Su:ras}

First we note a simple result on arbitrary rigid surjections. Let $T$ be an ordered tree. A non-empty set $T'\subseteq T$ is called a {\em subtree} if
it is closed downward with respect to $\sqsubseteq_T$, that is, if $w\in T'$, $v\in T$, and $v\sqsubseteq_T w$, then $v\in T'$.

\begin{lemma}\label{L:subt}
Let $S$, $T$ be ordered trees and let $f\colon T\to S$ be a rigid surjection. Let $T'$ be a subtree of $T$. Then $f[T']$ is a subtree
of $S$ and $f\res T'\colon T'\to f[T']$ is a rigid surjection.
\end{lemma}

\begin{proof} Let $i\colon S\to T$ be the injection of $f$. Let $w\in T'$ and let $v\in S$ be such that $v\sqsubseteq_S f(w)$. Since $i$ is an embedding
and since $i$ is an injection of $f$, we have
\[
i(v)\sqsubseteq_T i(f(w)) \sqsubseteq_T w.
\]
Thus, $i(v)\in T'$. Using again the fact that $i$ is the injection of $f$, we have
\[
v=f(i(v)) \in f[T'].
\]
So $f[T']$ is a subtree.

To check that $f\res T'\colon T'\to f[T']$ is a rigid surjection, note that since for $T'$ is closed downward with respect to $\sqsubseteq_T$ and
since $i(f(w))\sqsubseteq_Tw$ for $w\in T$, we have that $i[f[T']]\subseteq T'$. It is now obvious that $i\res f[T']\colon f[T']\to T'$ is an embedding
which is the injection of $f\res T'$.
\end{proof}

A rigid surjection $f\colon T\to S$ is called {\em sealed} if its injection maps the $\leq_S$-largest leaf of $S$ to the $\leq_T$-largest leaf of $T$.

For an ordered tree $S$ and $v\in S$, let
\begin{equation}\label{E:rest}
S^v = \{ w\in S\mid w\leq_S v\}.
\end{equation}
Note that this definition extends \eqref{E:cutt}.
It is clear that $S^v$ is closed under taking predecessors in $S$. We call trees of the form $S^v$, $v\in S$, {\em initial subtrees of $S$}.
If $f\colon T\to S$ is a rigid surjection and $v\in S$, then let
\begin{equation}\label{E:forres}
f^v = f\res T^{i(v)},
\end{equation}
where $i$ is the injection of $f$. We note the following lemma.

\begin{lemma}\label{L:ran}
Let $f\colon T\to S$ be a rigid surjection, let $i$ be its injection, and let $v\in S$. Then the domain of $f^v$ is $T^{i(v)}$ and the image
of $T^{i(v)}$ under $f^v$ is $S^v$, and $f^v$ is a sealed rigid surjection.
\end{lemma}

\begin{proof} By Lemma~\ref{L:subt}, only $S^v\subseteq f[T^{i(v)}]$ needs justifying. But note that for $w\in S^v$ we have $w\leq_S v$, so
$i(w)\in T^{i(v)}$, hence $w=f(i(w)) \in f[T^{i(v)}]$ as required.
\end{proof}

Or first aim, accomplished in Sections~\ref{Su:raP}--\ref{Su:lpP} is to prove the following proposition. Later, in Section~\ref{Su:setsu}, we show how to
derive Theorem~\ref{T:mn} from this proposition.

\begin{proposition}\label{P:almo}
Let $b>0$. Let $S, T$ be ordered trees. There is an ordered tree $V$ such that for each $b$-coloring of all
sealed rigid surjections from some $V^v$ to $S$, as $v$ varies over $V$, there is $v_0\in V$ and a sealed rigid surjection
$g\colon V^{v_0}\to T$ such that
\[
\{ f\circ g^t\mid f\colon T^t\to S\hbox{ a sealed rigid surjection},\, t\in T\}
\]
is monochromatic.
\end{proposition}

\subsection{Ramsey theoretic structures for Proposition~\ref{P:almo}}\label{Su:raP}

In this section, we describe concrete Ramsey theoretic structures of the kind defined in Sections~\ref{Su:ncs} and \ref{Su:rd} that are needed for
the proof of Proposition~\ref{P:almo}.

In the lemma below we record a simple observation about $f^v$.

\begin{lemma}\label{L:trde}
Let $f\colon T^w\to S$, $w\in T$, and $g\colon V\to T$ be rigid surjections. Let $i$ be the injection of $f$. Let $v\in S$. Then
\[
f^v\circ g^{i(v)} = (f\circ g^w)^v.
\]
\end{lemma}

\begin{proof} Let $j$ be the injection of $g$. It is clear from Lemmas~\ref{L:coin} and \ref{L:ran} that the domains of
both functions $f^v\circ g^{i(v)}$ and $(f\circ g^w)^v$ are equal to $V^{j(i(v))}$. For every $x$ in this set both functions
are equal to $f(g(x))$.
\end{proof}

Fix a family
\[
\mathcal T
\]
of ordered trees such that each ordered tree has an isomorphic copy in $\mathcal T$ and such that for $T_1, T_2\in {\mathcal T}$,
\[
T_1\cap T_2 =\emptyset.
\]
Let
\[
{\mathcal L} = \{ T^v\mid T\in {\mathcal T},\, v\in T \}.
\]
Introducing two families, $\mathcal T$ and $\mathcal L$, will be helpful in defining our Ramsey domain and
checking conditions (A) and (C) from the definition of Ramsey domain in this particular case.

We now define a normed composition space. Let $\A$ be the set of all sealed rigid surjections $g\colon T_2\to T_1$ for $T_1, T_2\in {\mathcal L}$.
The operation $\cdot$ is defined as follows. Let $f, g\in \A$. We let
$g\cdot f$ be defined precisely when $f\colon T^y\to S$ and $g\colon V\to T$ for some ordered trees
$S,T, V$ and a vertex $y$ in $T$. We let
\begin{equation}\label{E:compost}
g\cdot f= f\circ g^y.
\end{equation}
Note that the orders of $f$ and $g$ are different on the two sides of the equation above. Observe further that, by Lemma~\ref{L:ran}, the image of $g^y$ is equal to the domain of $f$.
The image of $g\cdot f$ is equal to $S$ and its domain is equal to the domain of $g^y$, that is, to $V^{j(y)}$, where $j$ is the
injection of $g$. So $g\cdot f\in \A$.

For $f\in \A$ whose image is an ordered tree $S$ define $\partial f$ as follows. If $S$ consists only of its root, let
\[
\partial f = f.
\]
If $S$ has a vertex that is not a root, let $v$ be the second $\leq_S$-largest vertex in $S$. Define
\[
\partial f = f^v.
\]

Consider $\mathcal L$ as a partial order with the partial order relation on it being inclusion. We make the following observation about the order of inclusion on $\mathcal L$.
By disjointness of $\mathcal T$, we have that for
$T_1, T_2\in {\mathcal L}$, $T_1\subseteq T_2$ precisely when there is $T\in {\mathcal T}$ and $v,w\in T$ such that $v\leq_T w$, $T_1 = T^v$, and
$T_2 = T^w$. We define $|\cdot|\colon \A\to {\mathcal L}$ by letting
\[
|f| = {\rm dom}(f)
\]
for $f\in \A$.

\begin{lemma}
The structure $(\A, \cdot, \partial, |\cdot|)$ defined above is a normed composition space.
\end{lemma}

\begin{proof}
Associativity of multiplication is clear from Lemma~\ref{L:coin}.

We check now the three axioms of normed composition spaces.
The identity $\partial(g\cdot f) = g\cdot \partial f$ is a special case of Lemma~\ref{L:trde} since this lemma implies
that for sealed rigid surjections $g\colon V\to T$ and $f\colon T^w \to S$, with $w\in T$, and for $v\in S$ we have
\[
(g\cdot f)^v = g\cdot f^v.
\]
Indeed, observe that $g\cdot f = f\circ g^w$ and $g\cdot f^v= f^v\circ g^{i(v)}$, where $i$ is the injection of $f$. Thus, we obtain the following sequence of equalities,
by using Lemma~\ref{L:trde} to get the second equality,
\[
(g\cdot f)^v = (f\circ g^w)^v = f^v\circ g^{i(v)} = g\cdot f^v.
\]

The second axiom, that is, the inequality $|\partial f|\leq |f|$, is clear from the definitions.

To check the third axiom, assume that $g\cdot f$ is defined. This means that $f\colon T^w\to S$ and $g\colon V\to T$. Moreover,
\[
|g\cdot f| = V^{j(w)},
\]
where $j$ is the injection of $g$. Now if $|f'|\leq |f|$, then $f'\colon T^v\to S'$ for some $v\in T$ with $v\leq_T w$.
Thus, $g\cdot f'$ is defined and
\[
|g\cdot f'| = V^{j(v)},
\]
which implies $|g\cdot f'|\leq |g\cdot f|$ as $j(v)\leq_V j(w)$.
\end{proof}

Now we define a Ramsey domain over $(\A, \cdot, \partial, |\cdot|)$. Recall the set $\mathcal T$ that was used to
defined $\mathcal L$ above.

Let $\mathcal F$ consist of non-empty sets $F\subseteq \A$ with the property that there are
$T_1, T_2\in {\mathcal T}$ such that for each $f\in F$, we have ${\rm rng}(f)= T_1$ and ${\rm dom}(f)\subseteq T_2$. Note that, since $f\in \A$ and $T_2\in {\mathcal T}$,
this last condition is equivalent to saying that ${\rm dom}(f)$ is an initial subtree of $T_2$.
It is possible for no function in $F$ as above to have its domain equal to $T_2$. Despite of this,
since the trees in $\mathcal T$ are pairwise disjoint, each $f\in F$ determines not only ${\rm dom}(f)$, but also $T_2$. Therefore,
it is possible to define
\[
d(F) = T_2\;\hbox{ and }\; r(F) =T_1.
\]
For $F_1, F_2\in {\mathcal F}$, let
$F_1\bullet F_2$ be defined precisely when $d(F_2) = r(F_1)$. Observe that in this case $f_1\cdot f_2$ is defined for all $f_1\in F_1$ and $f_2\in F_2$, and let
\[
F_1\bullet F_2 = F_1\cdot F_2.
\]
Note that $F_1\bullet F_2\in {\mathcal F}$ and
\[
d(F_1\bullet F_2) = d(F_1)\;\hbox{ and }\; r(F_1\bullet F_2) = r(F_2).
\]

Let ${\mathcal P}$ consist of all finite non-empty subsets $P$ of
$\A$ of the following form. There exist $S\in {\mathcal L}$ and $T\in {\mathcal T}$
such that for each $g\in P$, ${\rm rng}(g) = S$ and ${\rm dom}(g)\subseteq T$.  Let
\[
d(P)=T.
\]
So we have ${\mathcal F}\subseteq {\mathcal P}$.
For $F\in {\mathcal F}$ and $P\in {\mathcal P}$, $F\dbullet P$
is defined precisely when $d(P)= r(F)$, in which case, we let
\[
F\dbullet P = F\cdot P.
\]
Note that $f\cdot x$ is defined for each $f\in F$ and $x\in P$ and $d(F\dbullet P) = d(F)$. Furthermore, we have $F\dbullet P \in {\mathcal P}$.

\begin{lemma}\label{L:vanR}
The structure $({\mathcal F}, {\mathcal P}, \bullet,\dbullet)$ is a linear vanishing Ramsey domain over the composition space $(\A,\cdot,  \partial, |\cdot|)$.
\end{lemma}

\begin{proof} First we check in order conditions (A)--(C) from the definition of Ramsey domain.
Assume that, for $F_1, F_2\in {\mathcal F}$ and $P\in {\mathcal P}$,
$F_1\dbullet (F_2\dbullet P)$ is defined. Then $r(F_2) = d(P)$ and $r(F_1) = d(F_2\dbullet P)$. Since
$d(F_2\dbullet P) = d(F_2)$, we have $r(F_1) = d(F_2)$. It follows that $F_1\bullet F_2$ is defined and $r(F_1\bullet F_2) = r(F_2)$.
Thus, $(F_1\bullet F_2)\dbullet P$ is defined, as required by (A).
If $P\in {\mathcal P}$, then clearly $\partial P\in {\mathcal P}$, so (B) holds.
Note that if, for $F\in {\mathcal F}$ and $P\in {\mathcal P}$,
$F\dbullet \partial P$ is defined, then $F\dbullet P$ is defined since $d(\partial P) = d(P)$, and (C) follows.
We conclude that $({\mathcal F}, {\mathcal P}, \bullet,\dbullet)$ is a Ramsey domain.

If $P\in {\mathcal P}$ and $d(P) = T$, then
\[
\{ |f|\mid f\in P\} \subseteq \{ T^w\mid w\in T\}
\]
and the latter set is linearly ordered in $\mathcal L$. It follows that $({\mathcal F}, {\mathcal P}, \bullet,\dbullet)$ is linear.

Finally note that if $P\in {\mathcal P}$, $r(P)= S$, and $d(P)= T$, then, for the natural number $t$ equal to one less the number of vertices in $S$,
the range of each element of $\partial^t P$ is equal to the root of $S$.
Since these elements are sealed rigid surjections, it follows that the domain of each of them
also consists only of the root of $T$. Thus, there is precisely one such element. So, $({\mathcal F}, {\mathcal P}, \bullet,\dbullet)$ is vanishing.
\end{proof}

\subsection{Condition (LP) for Proposition~\ref{P:almo}}\label{Su:lpP}

It is clear that the conclusion of Proposition~\ref{P:almo} is just condition (R) for the Ramsey domain $({\mathcal F}, {\mathcal P}, \bullet,\dbullet)$ defined above.
So by Theorem~\ref{T:ram} in conjunction with Lemma~\ref{L:vanR}, to prove Proposition~\ref{P:almo}, it suffices to check condition (LP) for
$({\mathcal F}, {\mathcal P}, \bullet,\dbullet)$. This is what we will do in this section.

Sections~\ref{Su:rlp} and \ref{Su:slr} are, in a sense, preparatory. In Section~\ref{Su:rlp}, we find a condition which is equivalent to condition (LP) for our Ramsey domain but
has a form that makes it easier to prove. The basis of our arguments here is formed
by the construction of an ordered tree $(T; x_1, \dots , x_n)\oplus (T_1, \dots , T_n)$ out of an ordered tree $T$
and ordered forests $T_1, \dots, T_n$.
In Section~\ref{Su:slr}, we prove versions, appropriate for our goal of showing (LP), of auxiliary results stated earlier.

In Section~\ref{Su:lpr}, we give a proof of (LP), in which the main roles are played by the construction of an ordered forest $S\otimes I$ out of an ordered forest $S$ and a linear order $I$
and by particular rigid surjections, namely those fulfilling condition \eqref{E:spfu} of Section~\ref{Su:aux}.

\subsubsection{Restatement of (LP)}\label{Su:rlp}
To set up the formulation and the proof of condition (LP), we will need some new notions.
It will be convenient to use the notion of forest. By a {\em forest} we understand a finite partial order such that the set of
predecessors of each element is linearly ordered. The partial order relation on a forest $T$ is denoted by $\sqsubseteq_T$.
So a forest is a tree with the root removed. The following operation reverses this removal. For a forest $T$, let
\begin{equation}\label{E:onepl}
1\oplus T
\end{equation}
be the tree obtained from $T$ by adding to it one vertex with the vertex becoming the root of $1\oplus T$ and with $\sqsubseteq_T$ being the restriction to $T$
of the tree partial order $\sqsubseteq_{1\oplus T}$. We say that vertices $v_1, v_2$ of a forest $T$ are in the
same {\em component} if there is a vertex $w$ such that $w\sqsubseteq_T v_1$ and $w\sqsubseteq_T v_2$. Clearly, the components of a forest are
disjoint from each other and each of them is a tree.
A forest $T$ is an {\em ordered forest} if it is equipped with a linear order
relation, denoted by $\leq_T$, that is the restriction to $T$ of a linear order relation $\leq_{1\oplus T}$ on $1\oplus T$ that makes $1\oplus T$ into an ordered tree.
So $\leq_T$ is a linear order
that makes each component into an ordered tree and is such that each component of $T$ is an interval. A {\em tree embedding} from an ordered forest
$S$ to an ordered forest $T$ is a function from $S$ to $T$ that extends to an embedding from $1\oplus S$ to $1\oplus T$. Note that an embedding from $S$ to $T$
maps distinct components of $S$ to distinct components of $T$.

Let $T$ be an ordered tree, let $x_1, \dots, x_n\in T$ be distinct, and let $T_1, \dots, T_n$ be ordered forests. We define the ordered tree
\[
V= (T; x_1, \dots , x_n)\oplus (T_1, \dots , T_n)
\]
as follows. The set of all vertices of $V$ is the disjoint union of $T$ and $T_1,\dots , T_n$. The tree relation $\sqsubseteq_V$ on $V$ restricted to $T$
is $\sqsubseteq_T$ and restricted to each $T_i$ is $\sqsubseteq_{T_i}$. Further, for each $1\leq i\leq n$, $x_i\sqsubseteq_V v$ for $v\in T_i$ with the
minimal elements of $T_i$ being immediate successors of $x_i$. This description uniquely determines the tree relation on $V$. We make $V$ into an ordered
tree as follows. The linear order $\leq_V$ on $V$ when restricted to $T$ and $T_i$, $1\leq i\leq n$, is equal to $\leq_T$ and $\leq_{T_i}$, respectively.
Furthermore, we stipulate that $T_i$ is a final interval in the set $\{ v\in V\mid x_i\sqsubseteq_V v\}$ under $\leq_V$. This completely
describes $\leq_V$.
If $A$ is a non-empty linear order and $T$ is a forest, let
\[
A\oplus T = (A; \max A)\oplus (T).
\]
So this is the ordered tree obtained by putting $T$ on top of the linear order of $A$, and the tree is linearly ordered by putting the linear order of $T$ on top of $A$.
Note that if the forest order $\sqsubseteq_T$ is linear, then $A\oplus T$ is a linear order as well and the definition above coincides with the definition from
Section~\ref{Su:aux}. Recall that $A\oplus 1$ is $A\oplus T$, where $T$ consists of one element only. Similarly, if $A$ is a one element set, then $A\oplus T$ is
denoted by $1\oplus T$ as in \eqref{E:onepl}.

We discuss now condition (LP). In this condition we are given $P\in {\mathcal P}$, that is, we have
ordered trees $T\in {\mathcal T}$ and $S\in {\mathcal L}$ and a non-empty set $P$ of sealed rigid surjections from initial subtrees of $T$
onto $S$. We are also given $s_0\in \partial P$. We are looking for an appropriate $F\in {\mathcal F}$.
Note first that if $S$ has only one vertex, then, since elements of $P$ are sealed rigid surjections, $P$ has
only one element and $\partial P= P$, so (LP) is obvious in this case. Assume, therefore, that $S$ has at least two vertices.
Let $i_0$ be the injection of $s_0$.
Let $v_0, v_1\in S$ with $v_1<_S v_0$ be the two $\leq_S$-largest
vertices of $S$. Let $v_2 = v_0\wedge_S v_1$. Let also
\[
w_1 = i_0(v_1),\, w_2 = i_0(v_2)\in T.
\]
Since $s_0$ is sealed, its domain is $T^{w_1}$.

We need to produce
\begin{enumerate}
\item[(1)] an ordered tree $V\in {\mathcal T}$ and a non-empty set $F$ of sealed rigid surjections from initial subtrees of $V$ onto $T$, and

\item[(2)] an element $a\in \A$
\end{enumerate}
so that $F$ and $a$ fulfill (LP).

This will be done as follows. Let $x_1, \dots, x_n\in T$ list, in increasing order,
all $x\in T$ with $w_2\sqsubseteq_T x\sqsubseteq_T w_1$. For $1\leq i\leq n$, let $T_i$  be the forest
\[
T_i= \{ v\in T\mid x_i\sqsubseteq_T v,\, w_1 <_T v, \hbox{ and if }i<n, \hbox{ then }x_{i+1}\not\sqsubseteq_T v\}
\]
taken with the inherited tree relation and order relation. Let $T'$ be $T$ with
all the vertices in $T_1, \dots, T_n$ removed. So $T'$ is the union of $T^{w_1}$ and all the vertices $v\in T$ with
$w_2<_T v$ and $w_2\not\sqsubseteq_T v$. Note further that $T$ is isomorphic to
\[
(T'; x_1, \dots , x_n)\oplus (T_1, \dots , T_n).
\]

The ordered tree $V$ that we need to define will be an ordered tree in $\mathcal T$ isomorphic to an ordered tree of the form
\[
V= (T'; x_1, \dots , x_n)\oplus (V_1, \dots , V_n)
\]
for some ordered forests $V_1, \dots, V_n$ that will be specified later. We define $F$ to be the set of all rigid surjections from an initial
subtree of $V$ onto $T$. To define the element $a\in \A$,  let
\[
a={\rm id}_{T^{w_1}}.
\]
Since $T^{w_1}$ is an initial subtree of $V$, we indeed have $a\in \A$.
Note that $F\dbullet P$ and $a\cdot s_0$ are defined. It remains to specify $V_1, \dots, V_n$ and show that
for each $b$-coloring of $F_a\cdot P^{s_0}$ there is $f\in F_a$ such that $f\cdot P^{s_0}$ is monochromatic.

Let
\[
A_i = \{ w\in T\mid w\sqsubseteq_T x_i\}.
\]
The set $A_i$ is linearly ordered by $\sqsubseteq_T$.
Let
\[
B_i = s_0[A_i].
\]
Since $s_0$ is a rigid surjection, one readily checks that $B_i$ is linearly ordered and downwards closed under $\sqsubseteq_S$.
Further, since $x_1 = w_2=i_0(v_2)$, we have
\[
B_1 = \{ v\in S\mid v\sqsubseteq_S v_2\}.
\]

Now $P^{s_0}$ consists of all $s\in P$ with $s\colon T^w\to S$ for some $w\in T_1$ and such that $s\res T^{w_1} = s_0$.
Indeed, if $i$ is the injection of $s$, then, since $i$ is a morphism, we have $i(v_0)\wedge_T w_1 = w_2$ and, since $i$ is injective,
$i(v_0)\not= w_2$. So $i(v_0)\in T_1$. Since $s$ is a sealed rigid surjection, we get $s\colon T^{i(v_0)}\to S$ and we can take above $w=i(v_0)$.
Note that $T^w$ is
the disjoint union of $T^{w_1},\, T_1^w,\, T_2, \dots, T_n$. So each $s\in P^{s_0}$ is completely determined by $w\in T_1$ and
the restrictions
\[
s\res T_1^w, \, s\res T_2, \dots, \, s\res T_n.
\]
These restrictions are arbitrary functions with $s[T_i] \subseteq B_i$, for $2\leq i\leq n$, and with $s[T_1^w] \subseteq B_1\cup \{ v_0\}$
and $\{ w\}=s^{-1}(v_0)$.

On the other hand, $F_a$ consists of all sealed rigid surjections $t\colon V^y\to T$, for some $y\in V$ with $w_1\leq_V y$, with
$t\res T^{w_1} = {\rm id}_{T^{w_1}}$.
To witness (LP), we will only need those elements of $F_a$ that are of the form $t^w$, with $w\in T_1$, for some
rigid surjection $t\colon V\to T$ with $t\res T' = {\rm id}_{T'}$. Such a $t$ is completely determined by its restrictions
\[
t\res V_1, \dots, t\res V_n.
\]
Note that since $t$ is a rigid surjection, we have
\[
T_1\subseteq t[V_1] \subseteq A_1\cup T_1, \dots, T_n\subseteq t[V_n]\subseteq A_n\cup T_n.
\]

Therefore, (LP) boils down to proving the following statement.

Let $A_1, \dots, A_n$ and $B_1, \dots, B_n$ be non-empty linear orders.
Let $r_i\colon A_i\to B_i$ be a rigid surjection for $1\leq i\leq n$. Let $b>0$ be given. Assume $T_1, \dots, T_n$ are forests. There
exist forests $V_1, \dots, V_n$ with the following property. Assume we have a $b$-coloring of all sequences $(u_1, \dots , u_n)$ where
\begin{enumerate}
\item[---] $u_1\colon A_1\oplus V_1^y\to B_1\oplus 1$, for some $y\in V_1$, $u_i\colon A_i\oplus V_i\to B_i$, for $2\leq i\leq n$;

\item[---] $u_i\res A_i = r_i$, for $1\leq i\leq n$;

\item[---] $u_1$ is a sealed rigid surjection.
\end{enumerate}
Then there exist $t_i\colon A_i\oplus V_i\to A_i\oplus T_i$, for $1\leq i\leq n$, that are rigid surjections such that $t_i\res A_i = {\rm id}_{A_i}$ and the color assigned to
$(s_1\circ t_1^w, s_2\circ t_2, \dots, s_n\circ t_n)$ is fixed regardless of the choice of $(s_1, \dots, s_n)$ such that
\begin{enumerate}
\item[---] $s_1\colon A_1\oplus T_1^w\to B_1\oplus 1$, for some $w\in T_1$, $s_i\colon A_i\oplus T_i\to B_i$, for $2\leq i\leq n$;

\item[---] $s_i\res A_i = r_i$, for $1\leq i\leq n$;

\item[---] $s_1$ is a sealed rigid surjection.
\end{enumerate}

A moment's thought reveals that it suffices to show the above statement assuming that $B_i=A_i$, for all $1\leq i\leq n$, and that
each $r_i = {\rm id}_{A_i}$. With this in mind, we state now the condition that implies (LP) that we will prove in what follows. To make the statement
and the arguments that follow a bit more succinct, we adopt the following definition. A function $t\colon A\oplus T\to A\oplus S$, where $S$ and $T$
are ordered forest and $A$ a linear order, is called an {\em $A$-rigid surjection} if it is a rigid surjection and $t\res A = {\rm id}_A$. Note that in
the case when $S$ is the empty forest, an $A$-rigid surjection $t\colon A\oplus T\to A$ is simply a function such that $t\res A = {\rm id}_A$.

\smallskip

\noindent {\em Let $b>0$ be given. Let $A_1, \dots, A_n$ be non-empty linear orders, and let $T_1, \dots, T_n$ be ordered forests.
There exist ordered forests $V_1, \dots, V_n$ with the following property.
Assume we have a $b$-coloring of all tuples $(u_1, \dots, u_n)$, where $u_1\colon A_1\oplus V_1^y\to A_1\oplus 1$ is
a sealed $A_1$-rigid surjection,
with $y\in V_1$ depending on $u_1$, and each $u_i\colon A_i\oplus V_i\to A_i$, $2\leq i\leq n$, is an $A_i$-rigid surjection.
Then there exist $A_i$-rigid surjections $t_i \colon A_i\oplus V_i\to A_i\oplus T_i$, for $i\leq n$, such that all
\[
(s_1\circ t_1^w, s_2\circ t_2, \cdots, s_n\circ t_n)
\]
have the same color, where
$s_1\colon A_1\oplus T_1^w \to A_1\oplus 1$ is a sealed $A_1$-rigid surjection, $w\in T_1$, and $s_i\colon A_i\oplus T_i\to A_i$
is an $A_i$-rigid surjection, for $2\leq i\leq n$.}

\subsubsection{Adaptation of auxiliary lemmas from Sections~\ref{S:tree} and \ref{S:aR}}\label{Su:slr}

The following lemma is an immediate consequence of Lemma~\ref{L:prfrco}.

\begin{lemma}\label{L:frco}
Let $b>0$. Let $S$ be an ordered forest. There exists an ordered forest $S'$ such that for each $b$-coloring of vertices of $S'$
there is a tree embedding $i\colon S\to S'$ such that all elements of $i(S)$ have the same color.
\end{lemma}

Recall from Section~\ref{Su:aux} that, for linear orders $L$ and $I$, $L\times I$ is taken with the lexicographic order. Note also that property \eqref{E:spfu} from Section~\ref{Su:aux} implies that $p$ is an $A$-rigid surjection. Below
we will consider functions denoted by $p^x$, which, we recall, are defined by formula \eqref{E:forres}.

\begin{lemma}\label{L:HJ}
Let $b>0$. Let two linear orders $A$ and $L$ be given, with $A$ being non-empty. There is a linear order $I$ such that for each $b$-coloring
of all sealed $A$-rigid surjections from $A\oplus (L\times I)^y$ to $A\oplus 1$, where we allow $y$ to vary over $L\times I$, there is
\[
p\colon A\oplus  (L\times I) \to A\oplus  L
\]
with property \eqref{E:spfu} and such that for each given $x\in L$
\[
\{ r\circ p^x\mid r\colon A\oplus L^x\to A\oplus 1\hbox{ a sealed $A$-rigid surjection}\}
\]
is monochromatic, that is, the color of $r\circ p^x$ depends only on $x\in L$.
\end{lemma}

\begin{proof} We note that for each two linear orders $A$ and $J$, with $A$ non-empty, and $x\in J$, a sealed rigid surjection $s\colon A\oplus J^x \to A\oplus 1$
is uniquely determined by its restriction $s\res (A\oplus J)^{x-} \colon (A\oplus J)^{x-}\to A$, where $x-$ is the predecessor of $x$ in $A\oplus J$. It follows that Lemma~\ref{L:HJ}
is equivalent to Lemma~\ref{L:prHJ}.
\end{proof}

\begin{lemma}\label{L:HJpro}
Let $b>0$ and let $A_1, \dots, A_n$ and $L_1, \dots, L_n$ be linear orders, with $A_1, \dots, A_n$ non-empty.
There is a linear order $I$ with the following property.
Consider a $b$-coloring of $n$-tuples $(s_1, \dots, s_n)$ such that
\begin{enumerate}
\item[(i)] $s_1\colon A_1\oplus (L_1\times I)^y\to A_1\oplus 1$, for some $y\in L_1\times I$, is a sealed $A_1$-rigid surjection;

\item[(ii)] for $2\leq i\leq n$, $s_i\colon A_i\oplus (L_i\times I)\to A_i$ is an $A_i$-rigid surjection.
\end{enumerate}
Then there exist $p_i\colon A_i\oplus (L_i\times I) \to A_i\times L_i$, for $1\leq i\leq n$, with \eqref{E:spfu} such that for each
sealed $A_1$-rigid surjection $r_1\colon A_1\oplus L_1^x\to A_1\oplus 1$ and all $A_i$-rigid surjections $r_i\colon A_i\oplus L_i\to A_i$, for
$1\leq i\leq n$, the color of
\[
(r_1\circ p_1^x, r_2\circ p_2, \dots, r_n\circ p_n)
\]
depends only on $x$.
\end{lemma}

\begin{proof} Consider the product $A= A_n\times\cdots \times A_1$ with the lexicographic order. (In the argument below the choice of this order
is irrelevant.) Applying Lemma~\ref{L:HJ} to $b>0$, the order $A$, and the linear order $L_n\oplus \cdots \oplus L_1$, we get a linear order $I$ and
\[
p\colon A\oplus ( (L_n\oplus\cdots \oplus L_1)\times I) \to A\oplus L_n\oplus \cdots \oplus  L_1
\]
with property \eqref{E:spfu}. Note that we can canonically identify $(L_n\oplus\cdots \oplus L_1)\times I$ with
$(L_n\times I)\oplus \cdots \oplus (L_1\times I)$, which we do. With this identification, by \eqref{E:spfu}, we have $p(L_i\times I)\subseteq A\oplus L_i$.
Let, for $1\leq i\leq n$,
\[
\pi_i\colon A\oplus (L_n\oplus\cdots \oplus L_1)\to A_i\oplus  (L_n\oplus\cdots \oplus L_1)
\]
be the canonical projection. Now define $p_i\colon A_i\oplus (L_i\times I)\to A_i\oplus L_i$, $1\leq i\leq n$, by
\begin{equation}\notag
\begin{split}
p_i\res A_i &= {\rm id}_{A_i}\\
p_i\res (L_i\times I) &= (\pi_i\circ p)\res (L_i\times I).
\end{split}
\end{equation}
It is now routine to check that each $p_i$ has property \eqref{E:spfu} and that they fulfill the conclusion of the lemma.
\end{proof}

Finally, the following lemma is an immediate consequence of Lemma~\ref{L:prexmb}.

\begin{lemma}\label{L:exmb}
Let $A$ be a non-empty linear order and
let $S$ and $T$ be ordered forests. Let $i\colon S\to T$ be an embedding. There exits an  $A$-rigid surjection $s\colon A\oplus T\to A\oplus S$
such that the restriction of the injection of $s$ to $S$ is equal to $i$.
\end{lemma}

\subsubsection{Proof of (LP)}\label{Su:lpr}

In this section, we adopt the convention of identifying a natural number $n$ with the set of all its strict predecessors $\{ 0, \dots, n-1\}$;
in particular, $0=\emptyset$. A sequence $t$ of length $n$ is, for us, a function whose domain is $n= \{ 0, \dots, n-1\}$.
So, for a natural number $m\leq n$, $t\res m$ is the restriction of this function to $m$, and $t^\frown a$ is the extension of $t$ to a sequence of
length $n+1$ such that $(t^\frown a)\res n = t$ and $(t^\frown a)(n) = a$. For two sequences $t$ and $t'$, we write $t\subseteq t'$ if $t'$ extends $t$, that is, 
if $t'\res n = t$, where $n$ is the length of $t$. 

For a forest $T$ and $v\in T$, let ${\rm ht}_T(v)$ be the cardinality of the set of all predecessors of $v$ (including $v$), and let
\[
{\rm ht}(T) = \max \{ {\rm ht}_T(v)\mid v\in T\}.
\]
If $T$ is clear from the context, we suppress the subscript $T$ from ${\rm ht}_T(v)$.
Note that ${\rm ht}(v)=1$ precisely when $v$ is a minimal vertex of $T$.

Let $S$ be an ordered forest, and let $I$ be a finite set linearly ordered by $\leq_I$. As usual, we write $\sqsubseteq_S$ for the forest relation on $S$ and
$\leq_S$ for the linear order on $S$.
Set $n = {\rm ht}(S)$. Let
\[
S \otimes I = \{ (s,t)\in S\times I^{\leq n}\mid {\rm ht}(s) = |t|\},
\]
where $I^{\leq n}$ is the set of all sequences of elements of $I$ of length not exceeding $n$ and where $|t|$ denotes the length of the sequence $t$.

We introduce a binary relation on $S\otimes I$ as follows. For $(s_1,t_1), (s_2,t_2)\in S\otimes I$, let
\[
(s_1, t_1)\sqsubseteq_{S\otimes I} (s_2, t_2)
\]
if and only if, for $h = {\rm ht}(s_1)$, 
\begin{equation}\label{E:tro}
\begin{split}
s_1&\sqsubseteq_S s_2,\\
t_1\res (h-1) &= t_2\res (h-1),\hbox{ and }\\
t_1(h-1) &\leq_I t_2(h-1).
\end{split}
\end{equation}

\begin{lemma}
Let $S$ be a forest. Then $S\otimes I$ taken with $\sqsubseteq_{S\otimes I}$ is a forest. 
\end{lemma}

\begin{proof} The proof amounts to showing that $\sqsubseteq_{S\otimes I}$ is a partial order and that, for each $(s,t)\in S\otimes I$, the set 
\[
\{ (s',t')\mid(s',t') \sqsubseteq_{S\otimes I} (s,t)\}
\] 
is linearly ordered by $\sqsubseteq_{S\otimes I}$. All this is straightforward, and we leave it to the reader. 
\end{proof} 

We equip $S\otimes I$ with another binary relation $\leq_{S\otimes I}$ in order to turn $S\otimes I$ into an ordered tree. 
We define it first on the set of all immediate successors of each element of $S\otimes I$.
Let $(s,t)\in S\otimes I$ with $h={\rm ht}_S(s)$. The set of immediate successors of $(s,t)$ with respect to $\sqsubseteq_{S\otimes I}$ is 
\begin{equation}\label{E:imsu}
\bigl\{ \bigl(s', t^\frown (\min I)\bigr) \mid s' \in {\rm im}_S(s)\bigr\}\cup \bigl\{ \bigl(s, (t\res (h-1))^\frown i\bigr)\mid i \in {\rm im}_I(t(h-1))\bigr\}.
\end{equation} 
Note that the second set in the union above has one element if $t(h-1)<_I\max I$ and is empty if $t(h-1)=\max I$. For the elements 
of \eqref{E:imsu}, we set 
\begin{equation}\label{E:lio}
\bigl(s, (t\res (h-1))^\frown i\bigr)\leq_{S\otimes I} \bigl(s', t^\frown (\min I)\bigr)\leq_{S\otimes I} \bigl(s'', t^\frown (\min I)\bigr), 
\end{equation}
when $i\in {\rm im}_I(t(h-1))$ and $s', s'' \in {\rm im}_S(s)$ are such that $s'\leq_S s''$. This definition describes $\leq_{S\otimes I}$ 
on the sets of immediate successors of elements of $S\otimes I$. We extend it lexicographically using $\sqsubseteq_{S\otimes I}$ 
to a linear order on the whole set $S\otimes I$ as described in Section~\ref{Su:deot}. Thus, the following lemma is immediate. 

\begin{lemma}
Let $S$ be an ordered forest. Then $S\otimes I$ with $\sqsubseteq_{S\otimes I}$ 
and $\leq_{S\otimes I}$ is an ordered forest. 
\end{lemma}

We give a more explicit description of the order $\leq_{S\otimes I}$ below. 
Define $Q=Q(S,I)$ by letting
\begin{equation}\label{E:deq}
Q = \{ (s,u) \in S\times I^{<n}\mid {\rm ht}(s) = |u|+1\},
\end{equation}
where $n= {\rm ht}(S)$ and $I^{<n}$ is the set of all sequences of elements of $I$ whose length is strictly smaller than $n$.
For $s\in S$ with ${\rm ht}_S(s) = h$ and for $i<h$, we write
$s(i)$ for the unique vertex of $S$ such that $s(i)\sqsubseteq_Ss$ and ${\rm ht}_S(s(i)) = i+1$; thus producing a sequence $(s(0), \dots, s(h-1))$ with $s=s(h-1)$. 
With an element $(s,u)\in Q$ with $h={\rm ht}_S(s)$, we associate the sequence 
\begin{equation}\label{E:mix} 
\chi(s,u) = (s(0), u(0), s(1), u(1), \dots, s(h-2), u(h-2), s(h-1)). 
\end{equation} 
For $(s_1, u_1), (s_2, u_2)\in Q$, we let 
\[
(s_1, u_1)\leq_Q (s_2, u_2)
\]
if the sequence $\chi(s_1, u_1)$ precedes the sequence $\chi(s_2, u_2)$ in the lexicographic order arising from taking 
$S$ with $\leq _S$ and $I$ with $\leq^*_I$, the order reverse to $\leq_I$. So for $i,j\in I$, 
we set $i\leq^*_Ij$ precisely when $j\leq_Ii$.

The following lemma gives a description of $\leq_{S\otimes I}$ that will be useful in further considerations. 
\begin{lemma}\label{L:oiso}
The function 
\[
Q\times I\ni \bigl((s,u), i\bigr)\to \bigl(s, u^\frown i\bigr)\in S\otimes I
\]
is an isomorphism of linear orders if $Q\times I$ is taken with the lexicographic order arising from $\leq_Q$ and $\leq_I$ and $S\otimes I$ is taken 
with $\leq_{Q\otimes I}$. 
\end{lemma}

\begin{proof} Let $\pi\colon S\otimes I\to Q$ be defined by 
\[
\pi(s,t) = \bigl(s, t\res ({\rm ht}_S(s)-1)\bigr).
\]
Note that, for $(s,u)\in Q$, 
\begin{equation}\notag
\pi^{-1}(s,u) = \{ (s, u^\frown i)\mid i\in I\},
\end{equation}
and, by \eqref{E:tro}, the function 
\[
I\ni i\to (s, u^\frown i)\in S\otimes I 
\]
is an increasing injection from $(I, \leq_I)$ to $(S\otimes I, \sqsubseteq_{S\otimes I})$ and, therefore,
to the linear order $(S\otimes I, \leq_{S\otimes I})$. 
It follows that to get the conclusion of 
the lemma, it will suffice to show that $\pi$ is order preserving, that is, that for $(s_1, t_1), (s_2, t_2)\in S\otimes I$, 
\begin{equation}\label{E:pior}
(s_1, t_1)\leq_{S\otimes I} (s_2, t_2) \Longrightarrow \pi(s_1, t_1)\leq_Q \pi(s_2, t_2). 
\end{equation}

Checking \eqref{E:pior} is accomplished by verifying the following two implications: 
\begin{equation}\label{E:rere} 
(s_1, t_1)\sqsubseteq_{S\otimes I} (s_2, t_2) \Longrightarrow \pi(s_1, t_1)\leq_Q \pi(s_2, t_2)
\end{equation} 
and 
\begin{equation}\label{E:riri} 
\begin{split} 
\bigl( (s_1, t_1), (s_2, t_2)\in {\rm im}_{S\otimes I}(s,t),\, (&s_1, t_1)\leq_{S\otimes I} (s_2, t_2),\, (s_1, t_1)\sqsubseteq_{S\otimes I} (s_1', t_1') \bigr)\\
&\Longrightarrow \pi(s'_1, t'_1) \leq_Q \pi(s_2, t_2). 
\end{split} 
\end{equation}
We will be using formulas \eqref{E:tro}, \eqref{E:lio}, and \eqref{E:mix} without mentioning them explicitly. 

We show \eqref{E:rere} first. 
The condition $(s_1, t_1)\sqsubseteq_{S\otimes I} (s_2, t_2)$ implies 
that, for $h_1= {\rm ht}_S(s_1)$ and $h_2={\rm ht}_S(s_2)$, 
\[
s_1\sqsubseteq_S s_2\;\hbox{ and }\; t_1\res (h_1-1) \subseteq t_2\res (h_2-1),
\]
and, therefore, 
\[
\chi\bigl(\pi(s_1, t_1)\bigr)\subseteq \chi\bigl(\pi(s_2, t_2)\bigr),
\]
which gives $\pi(s_1, t_1)\leq_Q \pi(s_2, t_2)$, as required. 

We check now \eqref{E:riri}. Fix $(s,t)\in S\otimes I$, and set 
\[
h= {\rm ht}_S(s)
\]
for the rest of this proof. The elements of ${\rm im}_{S\otimes I}(s,t)$ are listed in \eqref{E:imsu}. 

First, we consider the case of $(s_1, t_1)$, $(s_2, t_2)$ such that 
\[
s_1, s_2\in {\rm im}_S(s), \, s_1\leq_S s_2,\, t_1=t_2 = t^\frown \min I, 
\]
We can assume $s_1<_S s_2$. From $(s_1, t_1)\sqsubseteq_{S\otimes I} (s_1', t_1')$, we get that 
\[
s_1\sqsubseteq_S s_1'\,\hbox{ and }\, t_1\res h = t_1'\res h, 
\]
which implies 
\[
\bigl(s(0), t(0) ,\dots, s(h-1), t(h-1), s_1\bigr) \subseteq \chi\bigl( \pi( s_1', t_1')\bigr).
\]
We also have 
\[
\bigl(s(0), t(0) ,\dots, s(h-1), t(h-1), s_2\bigr) \subseteq \chi\bigl( \pi( s_2, t_2)\bigr). 
\]
Thus, we get $\pi(s'_1, t'_1) \leq_Q \pi(s_2, t_2)$ since $s_1<_S s_2$. 

Now consider the case of $(s_1, t_1)$, $(s_2, t_2)$ such that 
\[
\begin{split} 
&s_1=s,\, t_1=  (t\res (h-1))^\frown i, \hbox{ where }i \in {\rm im}_I(t(h-1)),\\ 
&s_2\in {\rm im}_S(s),\, t_2 = t^\frown (\min I).
\end{split} 
\]
Let $(s_1', t_1')$ be such that $\bigl(s, (t\res (h-1))^\frown i\bigr)\sqsubseteq_{S\otimes I} (s_1', t_1')$, that is, 
\[
s\sqsubseteq_S s_1',\, t\res (h-1) = t_1'\res (h-1),\hbox{ and }i \leq_I t_1'(h-1). 
\]
It follows that 
\[
\bigl(s(0), t(0) ,\dots, s(h-2), t(h-2), s, t_1'(h-1)\bigr) \subseteq \chi\bigl(\pi\bigl(s_1', t_1'\bigr)\bigr), 
\]
and
\[
\bigl(s(0), t(0) ,\dots, s(h-2), t(h-2), s, t(h-1)\bigr) \subseteq \chi\bigl(\pi\bigl(s_2, t^\frown (\min I)\bigr)\bigr). 
\]
Thus, we get 
\[
\pi\bigl(s_1', t_1'\bigr) \leq_Q \pi\bigl(s_2, t^\frown (\min I)\bigr)
\]
since $t(h-1)<_I t_1'(h-1)$. Condition \eqref{E:riri}, and therefore also condition \eqref{E:pior}, is proved.
\end{proof}

For $(s,u) \in Q$, let
\begin{equation}\label{E:intv}
I(s,u) = \{ \left( s, u^\frown i\right) \mid i\in I\}.
\end{equation}
Note that, for $(s,u)\in Q$, $I(s,u)\subseteq S\otimes I$, the union $\bigcup_{(s,u)\in Q} I(s,u)$ is equal to $S\otimes I$, and, 
by Lemma~\ref{L:oiso}, $I(s,u)$ is an interval with respect to the linear order $\leq_{S\otimes I}$. 
At times, we will use the isomorphism from Lemma~\ref{L:oiso} to identify the linear order $Q\times I$ with $S\otimes I$ taken
with $\leq_{S\otimes I}$. Under this isomorphism $\{ (s,u)\}\times I$ is identified with $I(s,u)$.

In the lemma below, we will be
considering sealed $A$-rigid surjections $f$ from ordered trees of the form $A\oplus S$, where $S$ is an ordered forest, to $A\oplus 1$.
These are simply functions $f\colon A\oplus S\to A\oplus 1$ with the following two properties: $f\res A = {\rm id}_A$ and, for $s\in S$,
$f(s)\not\in A$ if and only if $s$ is the $\leq_S$-largest vertex in $S$. The lemma below is used to transfer the version
of the Hales--Jewett theorem from Lemma~\ref{L:HJ} to a Hales--Jewett--type theorem for trees.

\begin{lemma}\label{L:trtr}
Let $A$ be a non-empty linear order. Let $S$ be a forest and $I$ a linear order. Let $Q=Q(S,I)$.
Let
\[
p\colon A\oplus (Q\times I) \to A\oplus Q
\]
have property \eqref{E:spfu}. There is an $A$-rigid surjection
\[
\pi_p\colon A\oplus (S\otimes I)\to A\oplus S,
\]
with the following properties.

For every $v\in S$ there is $x\in Q$ such that for every sealed $A$-rigid surjection $\rho\colon A\oplus S^v \to A\oplus 1$,
there is a sealed $A$-rigid surjection $r\colon A \oplus Q^x \to A\oplus 1$ such that
\[
r\circ p^x = \rho\circ \pi^v_p,
\]
with the identification $Q\times I = S\otimes I$, so $A\oplus (Q\times I) = A\oplus (S\otimes I)$.

Similarly, for every $A$-rigid surjection $\rho\colon A\oplus S \to A$,
there is an $A$-rigid surjection $r\colon A \oplus Q \to A$ such that
\[
r\circ p = \rho\circ \pi_p.
\]
\end{lemma}

\begin{proof} Recall the definition \eqref{E:intv} of $I(s,u)$. Throughout this proof
we identify $Q\times I$ with $S\otimes I$ and $\{ (s,u)\}\times I$ with $I(s,u)$ for $(s,u)\in Q$.
Recall also that $p\colon A\oplus (Q\times I) \to A\oplus Q$ fulfills \eqref{E:spfu} if $p\res A ={\rm id}_A$ and, for each $(s,u)\in Q$,
\begin{equation}\label{E:clin} 
(s,u)\in p[I(s,u)]\subseteq A\cup \{ (s,u)\}.
\end{equation}

Fix $(s,t)\in S\otimes I$.
We say that $(s,t)$ is {\em leading} if it is the $\leq_{S\otimes I}$-smallest element of $I(s, t\res ({\rm ht}(s)-1))$
such that $p(s,t) = (s, t\res ({\rm ht}(s)-1))$. We call $(s,t)\in S\otimes I$ {\em very good} if each $(s',t')\in S\otimes I$
with $s'\sqsubseteq_S s$ and $t'\subseteq t$ is leading. We call $(s,t)$ {\em good} if
$p(s,t) = (s, t\res ({\rm ht}(s)-1))$ and each $(s',t')\in S\otimes I$
with $s'\sqsubseteq_S s$, $s'\not= s$, and $t'\subseteq t$, $t'\not= t$, is leading.

We claim that for each $s\in S$ there is exactly one $t$ such that $(s,t)$ is a very good element of $S\otimes I$. We show this by induction on
${\rm ht}(s)$. If ${\rm ht}(s)=1$, the conclusion is clear. Indeed, we take $t= \langle i\rangle$, where $i$ is the smallest element of $I$ with 
$p(s, \langle i\rangle) = (s, \emptyset)$, which exists by \eqref{E:clin}.
Obviously $(s,t)$ is very good and $t$ is unique such.
Let now ${\rm ht}(s)>1$ and let $s'$ be the immediate predecessor of $s$ in $S$. Let
$t'$ be the unique element such that $(s',t')$ is very good. Then $(s,t')\in Q$. Using \eqref{E:clin}, pick the smallest $i\in I$ such that $p(s, t'^\frown i) = (s,t')$. Then
$(s, t'^\frown i)$ is very good. It is clear that this $t'^\frown i$ is unique such.

For $s\in S$, the unique $t$ with $(s,t)$ very good will be denoted by $t_s$. Observe that for $s_1, s_2\in S$ with $s_1\sqsubseteq_S s_2$, we have
\begin{equation}\label{E:tris}
t_{s_1} = t_{s_2}\res {\rm ht}(s_1).
\end{equation}
Indeed, since $(s_1, t_{s_2}\res {\rm ht}(s_1))$ is very good, \eqref{E:tris} follows by uniqueness of $t_{s_1}$.
We also have for $(s,t)\in S\otimes I$
\begin{equation}\label{E:gogo}
\hbox{ if }(s,t) \hbox{ good, then } t_s\res ({\rm ht}(s)-1) = t\res ({\rm ht}(s)-1).
\end{equation}
Indeed, if $(s,t)$ is good, then $(s', t\res ({\rm ht}(s)-1))$ is very good, where $s'$ is the immediate $\sqsubseteq_S$-predecessor of $s$, so
$t_{s'} = t\res ({\rm ht}(s)-1)$, and \eqref{E:gogo} follows from \eqref{E:tris}.

Define $j_p\colon A\oplus S\to A\oplus (S\otimes I)$ by making it identity on $A$, and, for $s\in S$, letting
\[
j_p(s) = (s,t_s).
\]
It follows from \eqref{E:tris} and the definitions of $\sqsubseteq_{S\otimes I}$ and $\leq_{S\otimes I}$ that $j_p$ is an embedding. 

We define $\pi_p \colon A\oplus (S\otimes I)\to A\oplus S$ by making it identity on $A$ and, for $(s,t)\in S\otimes I$, letting
\begin{equation}\notag
\pi_p(s,t) = \begin{cases}
p(s,t), &\text{if $p(s,t) \in A$;}\\
s, &\text{if $(s,t)$ is good;}\\
\min A, &\text{if $p(s,t)\not\in A$ and $(s,t)$ is not good.}
\end{cases}
\end{equation}
Note that in the second case $p(s,t) = (s, t\res ({\rm ht}(s)-1))$.

We claim that $j_p$ is the embedding witnessing that $\pi_p$ is a rigid surjection.
Indeed, it is clear that $\pi_p\circ j_p = {\rm id}_{A\oplus S}$. It is also clear that $(j_p\circ \pi_p)\res A = {\rm id}_A$. It remains to verify that
for $(s,t)\in S\otimes I$ we have
\begin{equation}\label{E:emap}
j_p(\pi_p(s,t)) \sqsubseteq_{A\oplus (S\otimes I)} (s,t).
\end{equation}
So let $(s,t)\in S\otimes I$.
If $(s,t)$ is not good, then $\pi_p(s,t)\in A$, so $j_p(\pi_p(s,t))\in A$, and \eqref{E:emap} follows.
If $(s,t)$ is good, then, by \eqref{E:gogo},
\[
j_p(\pi_p(s,t)) = (s, (t\res ({\rm ht}(s)-1))^\frown i_0),
\]
where $i_0\in I$ is the smallest $i\in I$ such that
\[
p(s, (t\res ({\rm ht}(s)-1))^\frown i) = (s, t\res ({\rm ht}(s) -1)).
\]
Since, by virtue of $(s,t)$ being good, the value $p(s,t)$ is also $(s, t\res ({\rm ht}(s)-1))$, we get that $i_0\leq_I t({\rm ht}(s)-1)$, so 
\[
(s, (t\res ({\rm ht}(s)-1))^\frown i_0) \sqsubseteq_{A\oplus (S\otimes I)} (s,t).
\]
Thus, \eqref{E:emap} holds, as required.

Now we check the properties of $\pi_p$ claimed in the conclusion of the lemma. 
Let $v\in S$ be given. Define $x_v\in Q$ by letting
\[
x_v= (v, t_v\res ({\rm ht}(v)-1)).
\]
We write out the rest of the argument only for $\rho \colon A\oplus S^v\to A\oplus 1$; the
same formula defining $r$ works also in the case of $\rho\colon A\oplus S\to A$. 
So, let a sealed $A$-rigid surjection $\rho\colon A\oplus S^v\to A\oplus 1$ be given. We are looking for a sealed $A$-rigid surjection
$r\colon A\oplus Q^{x_v} \to A\oplus 1$ such that $r\circ p^{x_v} = \rho\circ \pi^v_p$. We let $r$ be identity on $A$.
For $(s,u)\in Q^{x_v}$, we define
\begin{equation}\notag
r(s,u) = \begin{cases} \rho(s), &\text{if there is $i\in I$ with $(s, u^\frown i)$ very good;}\\
\min A, &\text{if there is no $i\in I$ with $(s, u^\frown i)$ very good.}
\end{cases}
\end{equation}

We need to see that 
\begin{equation}\label{E:fie}
r\circ p^{x_v} = \rho\circ \pi_p^v. 
\end{equation}
Checking that, for $(s,t)\in S\otimes I$, if $r(p(s,t))$ and $\rho(\pi_p(s,t))$ are both defined, then they are equal,
boils down to an elementary case analysis, which follows the cases in the definition of $\pi_p$. This check involves 
the observation that, for $(s,t)\in S\otimes I$, $(s,t)$ is good if and only if $(s, t\res ({\rm ht}(s)-1)^\frown i)$ is very good for some $i\in I$. 
We leave the details to the reader. 

To finish proving \eqref{E:fie}, it remains to show that the domains of $p^{x_v}$ and $\pi_p^v$ are equal. 
This amounts to showing that the smallest, with respect to $\leq_{S\otimes I}$, element 
$(s,t)\in S\otimes I$ such that $\pi_p(s,t)=v$ is equal to the smallest $(s,t)\in S\otimes I$ such that 
$p(s,t)= (v, t_v\res ({\rm ht}(v)-1))$. We claim that both these conditions imply that $(s,t)= (v, t_v)$, which will finish the proof of the lemma. 

It suffices to see that either one of the two equations 
\[
\pi_p(s,t)=v, \;p(s,t)= (v, t_v\res ({\rm ht}(v)-1))
\]
implies that $s=v$ and $(s,t)$ is good, since then, by \eqref{E:tris} and Lemma~\ref{L:oiso}, the smallest such $(s,t)$ is very good, so $t=t_v$. 
Now, by definition of $\pi_p$, since $v\not\in A$, the condition $\pi_p(s,t) = v$ is equivalent to 
\[
s=v \hbox{ and } (s,t) \hbox{ is good}, 
\]
as required. By \eqref{E:clin}, the condition $p(s,t)= (v, t_v\res ({\rm ht}(v)-1))$ implies that 
\begin{equation}\label{E:rla}
s=v\;\hbox{ and }\;t= \bigl( t_v\res ({\rm ht}(v)-1)\bigr)^\frown i, \hbox{ for some }i\in I. 
\end{equation}
It follows that $p(v,t)= (v, t_v\res ({\rm ht}(v)-1))$, which, 
by \eqref{E:tris} and the second conjunct of \eqref{E:rla}, gives that $(v, t)$ is good. Thus, by the first conjunct of \eqref{E:rla}, $(s,t)$ is good, as required. 
\end{proof}

Now we prove condition (LP) as restated at the end of Section~\ref{Su:rlp}. Our notation is as in this statement.

For the given $b$ and $T_1$, Lemma~\ref{L:frco} produces an ordered forest $T_1'$. We claim that
\[
V_1= T_1'\otimes I, \, V_2=T_2\otimes I,\,\dots,\, V_n = T_n\otimes I
\]
for some linear order $I$ are as required.

Let $c$ be a $b$-coloring of all tuples $(u_1, \dots, u_n)$ as in the statement of (LP) with the above defined $V_1, \dots, V_n$. Let
\[
Q_1=Q(T_1',I),\, Q_2= Q(T_2, I),\, \dots,\, Q_n = Q(T_n,I)
\]
be defined as in \eqref{E:deq}. As usual, we identify $T_1'\otimes I$ with $Q_1\times I$ and $T_i\otimes I$ with $Q_i\times I$ for $2\leq i\leq n$.
Then $c$ extends to a coloring of all $n$-tuples whose entries are:
a sealed $A_1$-rigid surjection from $A_1\oplus (Q_1\times I)^y$ to $A_1\oplus 1$ for some $y\in Q_1\times I$ followed in order by
$A_i$-rigid surjections from $A_i\oplus (Q_i\times I)$ to $A_i$ for $2\leq i\leq n$ as in Lemma~\ref{L:HJpro}. By Lemma~\ref{L:HJpro}, there exists a linear order
$I$ and functions
\[
p_i\colon A_i\oplus (Q_i\times I) \to  A_i\oplus Q_i,
\]
for $i\leq n$, with property \eqref{E:spfu} and such that, for $x\in Q_1$ and a sealed $A_1$-rigid surjection $r_1\colon A_1\oplus (Q_1)^x\to A_1\oplus 1$ and $A_i$-rigid surjections
$r_i\colon A_i\oplus Q_i\to A_i$, for $2\leq i\leq n$,
the color
\begin{equation}\label{E:laco}
c(r_1\circ p_1^x, r_2\circ p_2, \dots, r_n\circ p_n)
\end{equation}
depends only on $x$.

Let now $\pi_{p_1}\colon A_1\oplus (T_1'\otimes I) \to A_1\oplus T_1'$ and
$\pi_{p_i}\colon A_i\oplus (T_i\otimes I) \to A_i\oplus T_i$, for $2\leq i\leq n$, be rigid surjections given by Lemma~\ref{L:trtr} applied to $p_1, p_2, \dots, p_n$.
It follows from Lemma~\ref{L:trtr} and the observation above that the color \eqref{E:laco} depends only on $x$
that, for $v\in T_1'$ and a sealed $A_1$-rigid surjection $s_1\colon A_1\oplus (T_1')^v\to A_1\oplus 1$ and $A_i$-rigid surjections
$s_i\colon A_i\oplus T_i\to A_i$, for $2\leq i\leq n$,
the color
\[
c(s_1 \circ \pi_{p_1}^v, s_2\circ \pi_{p_2}, \dots, s_n\circ \pi_{p_n})
\]
depends only on $v$. This observation gives a $b$-coloring of vertices $v$ of $T_1'$.
Let $i \colon T_1\to T_1'$ be an embedding such that $i[T_1]$ is monochromatic. By Lemma~\ref{L:exmb}, there exists a rigid surjection
$q\colon A_1\oplus T_1' \to A_1\oplus T_1$ whose injection restricted to $T_1$ is equal to $i$. Then
\[
q\circ \pi_{p_1}\colon A_1\oplus V_1 \to A_1\oplus T_1
\]
is a rigid surjection. Then
\[
t_1= q\circ \pi_{p_1}\;\hbox{ and }\;t_i = \pi_{p_i} \hbox{ for }2\leq i\leq n
\]
are as desired.

\subsection{Passage from sealed rigid surjections to arbitrary rigid surjections}\label{Su:setsu}

The aim of this section is to deduce Theorem~\ref{T:mn} from Proposition~\ref{P:almo}. The deduction is based on a new truncation-like operation for
rigid surjections that relies on the notion of conjugate leaves.

\subsubsection{Conjugate leaves and a truncation-like operation}\label{Su:lot}

By a {\em leaf} of a tree $T$ we understand a $\sqsubseteq_T$-maximal node of $T$. We write
\[
\ell(T)
\]
for the set of all leaves of $T$.
Let $S$ and $T$ be ordered trees. Let $i\colon S\to T$ be an embedding. We say that a leaf $y$ in $T$ is {\em $i$-conjugate} to a leaf
$x$ in $S$ provided that
\begin{enumerate}
\item[(i)] if $x$ is the $\leq_S$-largest leaf in $S$, then $y$ is the $\leq_T$-largest leaf in $T$;

\item[(ii)] if $x$ is not the $\leq_S$-largest leaf in $S$, let $x'$ be the $\leq_S$-smallest leaf with $x<_S x'$; then $y$ is the $\leq_T$-largest leaf in $T$
with
\begin{equation}\label{E:ccoo}
y<_Ti(x') \;\hbox{ and }\; i(x)\wedge_T i(x') = y\wedge_T i(x').
\end{equation}
\end{enumerate}

Note that in point (ii) above there always exists a leaf $y$ with \eqref{E:ccoo}; for example, any leaf $y$ with $i(x)\sqsubseteq_T y$ has this property.
We see that if $y$ is $i$-conjugate to $x$, then
\begin{equation}\notag
i(x)\leq_Ty<_T i(x').
\end{equation}
Note further that the set
\[
\{ y\in \ell(T) \mid i(x) \leq_T y<_T  i(x') \}
\]
contains two kinds of leaves---those for which $i(x)\wedge_Ti(x') = y\wedge_T i(x')$
and, possibly, those for which $i(x)\wedge_Ti(x') <_T y\wedge_T i(x')$. The leaves of the first kind form a non-empty $\leq_T$-initial segment of the set, and
the leaf $i$-conjugate to $x$ is the $\leq_T$-largest leaf in this segment.
Observe also that the $\leq_T$-largest leaf in $T$ is $i$-conjugate only to the $\leq_S$-largest leaf in $S$.

We drop the subscripts in $\wedge_S$, $\wedge_T$ and $\wedge_V$ in the subsequent proofs.

\begin{lemma}\label{L:cucu}
Let $i\colon S\to T$ and $j\colon T\to V$ be embeddings. Let $x\in \ell(S),\, y\in \ell(T)$ and $z\in \ell(V)$. Assume that $y$ is $i$-conjugate to $x$ and
$z$ is $j$-conjugate to $y$. Then $z$ is $(j\circ i)$-conjugate to $x$.
\end{lemma}

\begin{proof}
If one of the leaves $x, y, z$ is the largest leaf in its tree, then all of them are, and the conclusion of the lemma follows. We assume, therefore,
that $x,y,z$ are not the largest leaves in their trees.
We write $ji$ for $(j\circ i)$.

Let $x'$ be the $\leq_S$-smallest leaf in $S$ that is larger than $x$, and let $y'$ be the $\leq_T$-smallest leaf in $T$ that is larger than $y$.
Let
\[
A = \{ v\in \ell(V) \mid  ji(x)\wedge ji(x')<_V v\wedge ji(x')\},
\]
and let
\[
B = \{ v\in \ell(V) \mid  j(y)\wedge j(y')<_V v\wedge j(y')\}.
\]
Note that the immediate $\leq_V$-predecessor in $\ell(V)$ of the smallest point in $A$ is $ji$-conjugate to $x$, and the immediate $\leq_V$-predecessor in $\ell(V)$
of the smallest point in $B$ is $j$-conjugate to $y$. It suffices to show that the smallest
leaves in $A$ and $B$ are the same. Clearly $j(y')\in B$. Also note that by applying $j$ to $i(x)\wedge i(x')<_T y'\wedge i(x')$ we get that $j(y')\in A$. Thus,
it will be enough to show that
\begin{equation}\label{E:me}
A\cap \{ v\in \ell(V) \mid v\leq_V j(y')\} = B \cap \{ v\in \ell(V) \mid v\leq_V j(y')\}.
\end{equation}

First we make some observations about the relative position of $i(x),\, i(x'),\, y,$ and $y'$.
Note that since $y$ is $i$-conjugate to $x$,
\begin{equation}\label{E:prf}
i(x)\wedge i(x')\hbox{ is a strict $\sqsubseteq_T$-predecessor of }y'\wedge i(x').
\end{equation}
Note further that
\begin{equation}\label{E:numb}
i(x)\wedge i(x')= y\wedge i(x')  = y\wedge y'.
\end{equation}
Indeed, the first equality in \eqref{E:numb} follows immediately since $y$ is $i$-conjugate to $x$;
the second equality follows from the first one and from \eqref{E:prf}.

To show \eqref{E:me}, we need to prove two inclusions. We start with $\subseteq$. Using \eqref{E:numb}, note that
\begin{equation}\label{E:tr}
ji(x)\wedge ji(x') = j(y)\wedge j(y')
\end{equation}
Observe that $j(y')\leq_Vji(x')$ as $y'\leq_Ti(x')$. So, for $v\in \ell(V)$ with $v\leq_V j(y')$, we have $v\leq_V j(y')\leq_V ji(x')$, hence
$v\wedge ji(x') \sqsubseteq_V v\wedge j(y')$, and therefore
\begin{equation}\notag
v\wedge ji(x') \leq_V v\wedge j(y').
\end{equation}
From this inequality and from \eqref{E:tr}, it follows that $\subseteq$ holds in \eqref{E:me}.

To show the opposite inclusion, it suffices to see $B\subseteq A$. Assume that $v$ is a leaf in $V$ and $v\not\in A$, that is,
\begin{equation}\label{E:str}
v \wedge ji(x') \leq_V ji(x)\wedge ji(x').
\end{equation}
From it, since, by \eqref{E:prf}, $ji(x)\wedge ji(x')$ is a strict $\sqsubseteq_V$-predecessor of $j(y')\wedge ji(x')$,
we see that $v \wedge ji(x')$ is a strict $\sqsubseteq_V$-predecessor of
$j(y')\wedge ji(x')$. As a consequence, we immediately get
\begin{equation}\label{E:imr}
v\wedge ji(x') = v\wedge j(y').
\end{equation}
From \eqref{E:numb}, we have
\begin{equation}\label{E:pp}
ji(x)\wedge ji(x') = j(y)\wedge ji(x').
\end{equation}
From \eqref{E:numb} again we get
\begin{equation}\label{E:ll}
j(y)\wedge ji(x') = j(y)\wedge j(y').
\end{equation}
Putting together \eqref{E:imr}, \eqref{E:str}, \eqref{E:pp}, and \eqref{E:ll}, we get
\[
v\wedge j(y') \leq_V j(y)\wedge j(y').
\]
So $v\not\in A$ implies $v\not\in B$, and the lemma is proved.
\end{proof}

Let $f\colon T\to S$ be a rigid surjection. Let $x$ be a leaf in $S$. A leaf $y$ of $T$ is called {\em $f$-conjugate to $x$}
if $y$ is $i$-conjugate to $x$, where $i$ is the injection of $f$. For a leaf $x$ of $S$, define
\[
f_x = f\res T^y,
\]
where $y$ is the leaf in $T$ that is $f$-conjugate to $x$ and $T^{y}$ is defined by formula \eqref{E:rest} .

\begin{lemma}\label{L:ima}
Let $f\colon T\to S$ be a rigid surjection and let $x\in \ell(S)$. Then the image of $f_x$ is equal to $S^x$, and $f_x\colon T^y\to S^x$ is a rigid surjection,
where $y\in \ell(T)$ is $f$-conjugate to $x$.
\end{lemma}

\begin{proof} By Lemma~\ref{L:subt}, only $f[T^y]= S^x$ needs checking.
If $x$ is the $\leq_S$-largest leaf in $S$, the conclusion is clear. Assume therefore that $x$ is not the largest leaf.
Let $i$ be the injection of $f$, and let $x'$ be the $\leq_S$-smallest leaf in $S$ with $x<_S x'$.

To see $f[T^y]\subseteq S^x$, note that for $w\in T^y$ we have, by definition,
\begin{equation}\label{E:rata}
w\leq_T y
\end{equation}
and, as a consequence of the definition  and
of $y$ being $f$-conjugate to $x$,
\begin{equation}\label{E:qqq}
w\wedge i(x')\sqsubseteq_T w\wedge i(x).
\end{equation}

Now take $w\in T$ and assume that $f(w)\not\in S^x$. Then either $f(w)\sqsubseteq_S x'$ and $x\wedge x'$ is a strict $\sqsubseteq_S$-predecessor of $f(w)$, or
$x'<_S f(w)$. In the first case, we get that $i(f(w))\sqsubseteq_T i(x')$ and $i(x)\wedge i(x')$ is a strict $\sqsubseteq_T$-predecessor of $i(f(w))$.  Therefore, since
$i(f(w))\sqsubseteq_T w$, we get that $w\wedge i(x)$ is a strict $\sqsubseteq_T$-predecessor of $w\wedge i(x')$, contradicting \eqref{E:qqq}.
In the second case, we get
\[
y\leq_T i(x')<_T i(f(w)) \sqsubseteq_T w.
\]
So $y<_T w$ contradicting \eqref{E:rata}.

The inclusion $S^x\subseteq f[T^y]$ is clear: since $i(x)$ is in $T^y$ and $f(i(x)) = x$, we see that all leaves in $S^x$, and therefore all vertices of $S^x$, are in the image of $f\res T^y$.
\end{proof}

\begin{lemma}\label{L:compot}
Let $S,T,V$ be ordered trees, and let $g\colon V\to T$ and $f\colon T\to S$ be rigid surjections.
Let $x\in \ell(S)$, and let $y\in \ell(T)$ be $f$-conjugate to $x$. Then
\[
f_x\circ g_y = (f\circ g)_x.
\]
\end{lemma}

\begin{proof} Let $z$ be the leaf in
$V$ that is $g$-conjugate to $y$. Then we have
\[
f_x\circ g_y = (f\res T^y)\circ (g\res V^z) = (f\circ g)\res V^z,
\]
where the last equality holds as $g[V^z]\subseteq T^y$ by Lemma~\ref{L:ima}. Since, by Lemmas~\ref{L:cucu} and \ref{L:coin}, we have that $z$ is $(f\circ g)$-conjugate to $x$, we have
\[
(f\circ g)_x = (f\circ g)\res V^z,
\]
and the lemma follows.
\end{proof}

\subsubsection{Proof of Theorem~\ref{T:mn} from Proposition~\ref{P:almo}}

Fix a natural number $b>0$ and ordered trees $S$ and $T$ as in the assumption of Theorem~\ref{T:mn}. Let $s$ and $t$ be the largest vertices
in $S$ and $T$ with respect to $\leq_S$ and $\leq_T$, respectively.
Let $S^+$ be the ordered tree obtained from $S$ by adding one vertex $s^+$ so that $s^+$ is an immediate $\sqsubseteq_{S^+}$-successor of the root and it is
the $\leq_{S^+}$-largest element of $S^+$. Let $T^+$ be an ordered tree obtained from $T$ in an analogous way by adding one vertex $t^+$.
Note that each rigid surjection $f\colon T\to S$ extends to a sealed rigid surjection $f'\colon T^+\to S^+$ by mapping $t^+$ to $s^+$, and observe that
\begin{equation}\label{E:ccc}
t\hbox{ is }f'\hbox{-conjugate to } s\;\hbox{ and }\; (f')_s =f.
\end{equation}

Let $U$ be an ordered tree obtained from Proposition~\ref{P:almo} for $b$, $S^+$ and $T^+$. We claim that the following statement holds.

\noindent {\em For each $b$-coloring of all rigid surjections from $U^y$ to $S$, where $y\in \ell(U)$, there exists $y_0\in \ell(U)$ and a rigid surjection
$g\colon U^{y_0}\to T$ such that the set
\[
\{ f\circ g\mid f\colon T\to S\hbox{ a rigid surjection}\}
\]
is monochromatic.}

Indeed, assume we have a $b$-coloring $c$ as in the assumption of the statement. We define now a $b$-coloring
$c'$ of all sealed rigid surjections from $U$ to $S^+$ as follows. For a sealed rigid surjection $h\colon U\to S^+$, let
\[
c'(h) = c(h_s).
\]
By our choice of $U$, there exists a sealed rigid surjection $g^+\colon U\to T^+$ such that the color $c'(f'\circ g^+)$ is fixed for
all sealed rigid surjections $f'\colon T^+\to S^+$.  Let $y_0\in \ell(U)$ be $g^+$-conjugate to $t$ and let $g = (g^+)_t$. Then
$g\colon U^{y_0}\to T$ is a rigid surjection. We show that it is as required by the conclusion of the statement.
If $f\colon T\to S$ is a rigid surjection, let $f'\colon T^+\to S^+$ be the sealed rigid
surjection obtained by mapping $t^+$ to $s^+$. Then, using Lemma~\ref{L:compot} and \eqref{E:ccc}, we obtain
\[
c(f\circ g) = c((f')_s \circ (g^+)_t) = c((f'\circ g^+)_s) = c'(f'\circ g^+).
\]
Thus, the color $c(f\circ g)$ does not depend on $f$.

We deduce the conclusion of Theorem~\ref{T:mn} from the above statement. We need to produce an ordered tree $V$. Let $U$ be as in
the conclusion of the statement above. For $y\in \ell(U)$, let $U_0^y$ be the ordered forest obtained from the ordered tree $U^y$ by removing the root. Let
$V_0$ be the ordered forest whose underlying set is the disjoint union $\bigcup_{y\in \ell(U)} U_0^y$, whose forest relation $\sqsubseteq_{V_0}$
is equal to $\sqsubseteq_{U_0^y}$ when restricted
to $U_0^y$ and does not relate vertices from distinct sets $U_0^y$, and whose linear order relation $\leq_{V_0}$
is equal to $\leq_{U_0^y}$ when restricted to
$U_0^y$ and makes all vertices in $U_0^{y}$ $\leq_{V_0}$-smaller than all vertices in $U_0^{y'}$ if $y<_U y'$. Finally, let
$V = 1\oplus V_0$, where the right hand side is defined as in the beginning of Section~\ref{Su:rlp}. We consider each $U^y$ to be a subtree of $V$ consisting
of $U_0^y$ and the root of $V$.

We claim that the ordered tree $V$ is as required. For each $y\in \ell(U)$, let
\[
\pi_y\colon V\to U^y
\]
be defined by letting $\pi_y\res U^y = {\rm id}_{U^y}$ and by mapping each $U^{y'}$ to the root of $U^y$ for $y'\not= y$. Note that $\pi_y$ is a rigid surjection; its
injection is ${\rm id}_{U^y}$. Now assume we have a $b$-coloring $c$ of all rigid surjections from $V$ to $S$. Define a $b$-coloring $c'$ of all rigid surjections from
$U^y$ to $S$ for $y\in \ell(U)$ by letting for $f\colon U^y\to S$
\[
c'(f) = c(f\circ \pi_y).
\]
It follows from the statement that there exists $y_0\in \ell(U)$ and a rigid surjection $g'\colon U^{y_0}\to T$ such that the color $c'(f\circ g')$ does not depend on the rigid surjection
$f\colon T\to S$. Define now a rigid surjection $g\colon V\to T$ by
\[
g = g'\circ \pi_{y_0}.
\]
Note that if $f\colon T\to S$ is a rigid surjection, then
\[
c(f\circ g) = c(f\circ g'\circ \pi_{y_0}) = c'(f\circ g')
\]
so the color $c(f\circ g)$ does not depend on $f$ as required, and Theorem~\ref{T:mn} is proved.

\medskip

\noindent {\bf Acknowledgement.} I would like to thank Miodrag Soki{\'c} and Anush Tserunyan for their 
remarks concerning the paper.

\end{document}